\pdfoutput=1
\documentclass[10pt]{amsart}

\usepackage[bookmarks,colorlinks=false,backref=false,plainpages=false]{hyperref}
\usepackage{geometry}
\usepackage{amsmath,amssymb,amsthm}
\usepackage{graphicx}
\graphicspath{{figs/}}
\usepackage{placeins}

\newtheorem{thm}{Theorem}[section]
\newtheorem{cor}[thm]{Corollary}
\newtheorem{lem}[thm]{Lemma}
\newtheorem{prop}[thm]{Proposition}
\theoremstyle{definition}
\newtheorem{assum}[thm]{Assumption}
\theoremstyle{remark}
\newtheorem{rem}[thm]{Remark}
\theoremstyle{plain}

\usepackage{amsaddr}

\numberwithin{equation}{section}

\renewcommand\o{\omega}
\renewcommand\O{\Omega}
\def\G{\Gamma}
\def\p{\partial}
\def\half{\frac{1}{2}}
\newcommand{\bsigma}{\mbox{\boldmath$\sigma$}}
\newcommand{\btau}{\mbox{\boldmath$\tau$}}
\newcommand{\bchi}{\mbox{\boldmath$\chi$}}
\newcommand{\bbeta}{\mbox{\boldmath$\beta$}}
\def\bb{{\bf b}}
\def\bff{{\bf f}}
\def\bn{{\bf n}}
\def\bx{{\bf x}}
\def\by{{\bf y}}
\def\bzero{{\bf 0}}
\def\bRT{{\mathbf{RT}}}
\def\cC{{\mathcal{C}}}
\def\cE{{\mathcal{E}}}
\def\cT{{\mathcal{T}}}
\newcommand{\gradt}{\nabla\cdot}
\def\divvr{\operatorname{div}}
\newcommand{\tr}{\operatorname{tr}}
\DeclareMathOperator{\osc}{osc}
\newcommand{\tri}{|\!|\!|}
\newcommand{\beq}{\begin{equation}}
\newcommand{\eeq}{\end{equation}}
\def\H00{H^{1/2}_{00}}
\def\Hm00{H^{-1/2}_{00}}

\title[LSFEM with weakly imposed boundary conditions]{Least-Squares Finite Element
Methods with Weakly Imposed Boundary Conditions: Norm and Error Equivalences,
Over-Penalization, and Fully Discontinuous Discretizations}
\author[S. Zhang]{Shun Zhang}
\address{Department of Mathematics, City University of Hong Kong, Kowloon Tong, Hong Kong, China}
\email{shun.zhang@cityu.edu.hk}
\thanks{This work was supported in part by
Research Grants Council of the Hong Kong SAR, China, under the GRF Grant Project No. CityU 11316222, CityU 11305025}
\date{\today}

\begin{document}

\begin{abstract}
Weakly imposed boundary conditions in first-order system least-squares methods
go back to Jespersen (1977), but two questions have remained open for decades:
the right
weights, especially for the Neumann residual, whose natural norm $H^{-1/2}$
suggests $h_F$, and norm and error equivalences, especially for indefinite
equations.  We answer both for general, possibly indefinite, second-order
elliptic equations with non-homogeneous mixed boundary conditions.  First, for
the weights $h_F^{-1}$ on the Dirichlet and $1$ on the Neumann residual, we
prove norm and error equivalences with mesh-independent constants
for every discrete pair: the error in a mesh-dependent norm is equivalent to
the square root of the functional plus a Dirichlet data term.  The
least-squares solution converges optimally in
$H(\mathrm{div};\Omega)\times H^{1}(\Omega)$, and the functional is a reliable
and efficient error estimator, with a recovered flux also for approximations
from other methods, even ones violating the boundary conditions.  The key idea
is to apply the stability of the boundary value problem to pairs that are
discrete only in their Dirichlet traces, whose $H^{1/2}$ norm an inverse
inequality bounds by the weighted penalty.  Second, these weights are the right
ones: weaker powers, including $h_F$ on the Neumann boundary, destroy the norm
equivalence, and stronger weights are admissible, since the penalty multiplies
zero at a comparison function meeting the projected data.  Third, for the fully
discontinuous counterpart, a pure penalty method with both jumps weighted by
$h_F^{-1}$, the error equivalence holds for every discrete pair, also for exact
fluxes with normal traces not square integrable on interior faces.
\end{abstract}

\maketitle

\section{Introduction}
\label{sec:intro}

First-order system least-squares methods recast a second-order elliptic
equation as a first-order system and minimize a sum of squared residuals
\cite{CLMM:94}.  In the resulting least-squares finite element method (LSFEM)
the value of the functional is a built-in a posteriori error estimator, not
only for the least-squares solution but for every conforming approximation
that satisfies the boundary conditions.  This rests on the coercivity of the
functional.  For general, possibly indefinite, equations
it was obtained through compactness arguments \cite{BLP:97,Cai:04,CFZ:15},
through stability estimates of the boundary value problem \cite{Ku:07}, and
directly from the inf--sup stability of the variational problem under the
single assumption that the homogeneous problem has only the trivial solution
\cite{Zhang:23}, where right-hand sides in $H^{-1}(\O)$ are absorbed into the
flux.  This paper concerns the non-homogeneous mixed boundary conditions
$u=\phi$ on $\G_D$ and $\bsigma\cdot\bn=\psi$ on $\G_N$, with $\bsigma$ the
flux.

We add the boundary residuals to the functional in face-wise weighted $L^{2}$
norms,
$$
\sum_{F\in\cE_D}h_F^{-1}\|v-\phi\|_{0,F}^{2}+\|\btau\cdot\bn-\psi\|_{0,\G_N}^{2},
$$
so that the trial spaces carry no boundary constraints and the exact data,
which need only satisfy $\phi\in H^{1/2}(\G_D)$ and $\psi\in L^{2}(\G_N)$,
enter directly.

Imposing the boundary conditions weakly is long known.  Least-squares methods
whose trial functions need not satisfy the boundary conditions go back to
Bramble and Schatz \cite{BS:71}.  For the first-order system, Jespersen
\cite{Jes:77} added the Dirichlet residual with the weight $h^{-1}$ for the
Poisson equation and analyzed the method in the mesh-dependent norm induced by
the functional, and Bensow and Larson \cite{BL:05} used the weights $h_F^{-1}$
on $\G_D$ and $1$ on $\G_N$, together with $h_F^{-1}$ on the interior jumps,
for the diffusion equation.  Two questions have nevertheless remained open for
decades.  The first is which weights are right.  The natural norms of the
boundary residuals are those of $H^{1/2}(\G_D)$ and $H^{-1/2}_{00}(\G_N)$, and
their scaling suggests the weights $h_F^{-1}$ and $h_F$; for the Neumann
residual in particular, the right weight has been unclear.  The second is a
norm equivalence, and with it an error equivalence, in particular for
indefinite equations.  The argument of \cite{BL:05} relies on the coercivity
of the weak form and does not extend to indefinite equations, it assumes that
the boundary data are finite element functions, and it gives coercivity only
on the finite element space; in \cite{BL:05} the functional serves as a
refinement indicator without a proof of reliability, and the fully
discontinuous variant \cite[Sect.~3.3.1]{BL:05} is stated without analysis.

Much of the literature has instead imposed the boundary conditions in the
natural norms: through multilevel boundary functionals \cite{Sta:99}, through
additional equations measured in the natural norms \cite{Ste:14}, or through
dual norms that lead to saddle-point and minimax problems \cite{MSSte:26}.  For
$\phi\in H^{1/2}(\G_D)$ and $\psi\in L^{2}(\G_N)$ none of these constructions
is needed: the weights $h_F^{-1}$ on $\G_D$ and $1$ on $\G_N$ are the right
ones, and they give the norm equivalence for
general, possibly indefinite, equations, an optimal a priori estimate, and a
reliable and efficient error estimator.

The paper makes three contributions.  The first and main one is the norm
equivalence and the error equivalence, with mesh-independent constants, for
general, possibly indefinite, equations; these are the first such
equivalences for boundary conditions weakly imposed in weighted $L^{2}$
norms.  They hold for every discrete pair, whatever method
produced it, and not only for the least-squares solution.  Since the functional
depends on the mesh, a compactness argument does not give mesh-independent
constants; following \cite{Zhang:23}, we derive the lower bound from the
stability of the boundary value problem and bound the natural norm of each
boundary residual by its weighted $L^{2}$ norm.  On $\G_N$ this is the
embedding $L^{2}(\G_N)\hookrightarrow H^{-1/2}_{00}(\G_N)$.  The key idea
concerns $\G_D$: we prove the norm equivalence for all pairs whose Dirichlet
traces lie in the space $W^{D}_{h}$ of Dirichlet traces of the discrete
scalars, where an inverse inequality bounds the $H^{1/2}(\G_D)$ norm by the
weighted $L^{2}$ norm.  This space is discrete only in the Dirichlet trace and
contains every discrete pair.

Because the lower bound rests on the stability of the boundary value problem,
the same argument applies to the error.  The Dirichlet trace $\phi-v$ of the
error is not discrete; it is split into $\phi-Q_h\phi$, which depends only on
the datum, and the discrete trace $Q_h\phi-v$, to which the inverse inequality
applies.  This gives the error equivalence: for every pair in this space, the error in the norm induced by the functional is
equivalent to the square root of the functional plus the best approximation
error of $\phi$ in $H^{1/2}(\G_D)$ by $W^{D}_{h}$, a term that depends only on
the Dirichlet datum.  Hence the functional is a
reliable and efficient error estimator for the least-squares solution and
equally for finite element approximations computed by other methods, such as
the Galerkin method with interpolated Dirichlet data; for these, a flux is
recovered by minimizing the functional with the scalar fixed, as in the
non-intrusive estimator of \cite{LiZhang:25}.  In particular, the a posteriori
estimator no longer requires the approximation to satisfy the boundary
conditions, and the Dirichlet term of the functional is the computable
violation of the Dirichlet condition.  For the least-squares solution, the
norm equivalence and a Raviart--Thomas comparison flux that meets the projected
Neumann data exactly give an optimal a priori estimate in
$H(\divvr;\O)\times H^{1}(\O)$.

The second contribution settles the weights: they are sharp, and any larger
weights are admissible.  For the weights $h_F^{-\alpha}$ on $\G_D$ and
$h_F^{-\beta}$ on $\G_N$, the norm equivalence holds with the same constants
for $\alpha\ge1$ and $\beta\ge0$ and fails for $\alpha<1$ or $\beta<0$, already
on quasi-uniform meshes; in particular, it fails for the weight $h_F$ on $\G_N$
suggested by the scaling of $H^{-1/2}_{00}(\G_N)$.  Larger weights leave the a priori estimate unchanged
up to the approximation of $\phi$ in $W^{D}_{h}$: a comparison function meets
the projected boundary data exactly, so every weight multiplies zero.  Thus the unweighted $L^{2}(\G_N)$ penalty, although
stronger than the natural $H^{-1/2}_{00}(\G_N)$ norm, leaves the a priori
estimate optimal.  The mechanism is not specific to finite elements: on $\G_N$
the weight $1$ suffices for every trial class, and in every least-squares
minimization a weight above the stability threshold can be increased without
changing the error bound whenever the trial class contains an accurate
comparison function that meets the boundary data.

The third contribution is the error equivalence for the fully discontinuous
counterpart, with both interior jumps weighted by $h_F^{-1}$: for every
discrete pair, the error in the broken norm is equivalent to the square root of
the functional plus the same Dirichlet data term, also for exact fluxes whose
normal traces on interior faces are not square integrable.  The proof needs no
norm equivalence on broken Sobolev spaces and uses the exact solution as
intermediary: the jumps bound the distance from a discrete pair to its
conforming companions, and the error equivalence of the first contribution
compares the companions with the exact solution.  This averaging argument was
introduced in \cite[Thm.~7.3]{LiangZhang:23}; Bringmann follows it in
\cite[Thm.~12]{Bri:26} for the Poisson problem with homogeneous boundary
conditions.  On the broken spaces without further constraints, $h_F^{-1}$ is
the weakest weight for either jump, and larger weights leave the rate
unchanged, as in the weakly over-penalized interior penalty methods
\cite{BOS:08,BOS:12}.  Bringmann \cite{Bri:26} obtains the weight $h_F$ on the
flux jump only under the extra constraint that the lowest moments of the normal
jumps vanish, which complicates the computation.  Our method is a pure penalty
method with no extra mechanism, and the error equivalence and the optimal a
priori estimate hold without this constraint.

Section~\ref{sec:prelim} collects the notation and the finite element spaces,
and Section~\ref{sec:bvp} the well-posedness of the first-order system.
Section~\ref{sec:functional} proves the norm and error equivalences,
Sections~\ref{sec:apriori} and~\ref{sec:aposteriori} the a priori and a
posteriori estimates for the least-squares solution,
Section~\ref{sec:nonintrusive} the non-intrusive a posteriori estimate for
approximations computed by other methods, which may violate the boundary conditions,
Section~\ref{sec:weights} the sharpness of the weights
and over-penalization, and Section~\ref{sec:dlsfem} the results for the fully
discontinuous method.  Section~\ref{sec:numerics} presents numerical
experiments.

\section{Notation, trace spaces and finite element spaces}
\label{sec:prelim}

\subsection{Domain, coefficients and trace spaces}
Let $\O\subset\mathbb R^{d}$, $d=2,3$, be a bounded, connected Lipschitz
polyhedron with $\p\O=\overline{\G_D}\cup\overline{\G_N}$, where $\G_D$ and
$\G_N$ are disjoint, relatively open, of positive $(d-1)$-dimensional measure,
and, for $d=3$, have Lipschitz relative boundaries.  We assume that
$A\in L^{\infty}(\O;\mathbb R^{d\times d})$ is symmetric with
$\Lambda_0|\by|^2\le\by^{T}A\by\le\Lambda_1|\by|^2$, that
$\bb,\bbeta\in L^{\infty}(\O)^{d}$, and that $\mu\in L^{\infty}(\O)$.

Let $\gamma\subset\p\O$ be relatively open with positive measure, let
$\tr^{g}:H^{1}(\O)\to H^{1/2}(\p\O)$ be the trace operator,
$\tr^{g}_{\gamma}(v):=\tr^{g}(v)|_{\gamma}$, and let $H^{1}_{D}(\O)$ and
$H^{1}_{\p\O\setminus\gamma}(\O)$ be the subspaces of $H^{1}(\O)$ with
vanishing trace on $\G_D$ and on $\p\O\setminus\gamma$.  We use the quotient
norms
\beq\label{def_quotient}
\|w\|_{1/2,\p\O}:=\inf_{\tr^{g}(v)=w}\|v\|_{1,\O},\qquad
\|w\|_{1/2,\gamma}:=\inf_{g|_{\gamma}=w}\|g\|_{1/2,\p\O}
\eeq
on $H^{1/2}(\p\O)$ and $H^{1/2}(\gamma):=\{g|_{\gamma}:\ g\in H^{1/2}(\p\O)\}$, so that
\begin{align}
\label{eq:quotient_inf}
\|w\|_{1/2,\gamma}&=\inf\{\|v\|_{1,\O}:\ v\in H^{1}(\O),\ \tr^{g}_{\gamma}(v)=w\},\\
\label{eq:quotient_lift}
\|\tr^{g}_{\gamma}(v)\|_{1/2,\gamma}&\le\|v\|_{1,\O}\qquad\forall v\in H^{1}(\O).
\end{align}
The
Lions--Magenes space $\H00(\gamma)$ consists of the traces on $\gamma$ of the
functions in $H^{1}_{\p\O\setminus\gamma}(\O)$ \cite[p.~121]{sayas:2019book},
with the norm
\beq\label{def_norm00}
\|v\|_{\H00(\gamma)}:=\inf\{\|w\|_{1,\O}:\ w\in H^{1}_{\p\O\setminus\gamma}(\O),\
\tr^{g}_{\gamma}(w)=v\},
\eeq
and $\Hm00(\gamma)$ is its dual, so that
\beq\label{bd_on_12}
\langle q,v\rangle_{\Hm00(\gamma)\times\H00(\gamma)}
\le\|q\|_{\Hm00(\gamma)}\|w\|_{1,\O}
\eeq
for every $w\in H^{1}_{\p\O\setminus\gamma}(\O)$ with $\tr^{g}_{\gamma}(w)=v$.
For $\btau\in H(\divvr;\O)$ the normal trace $\tr^{d}_{\gamma}(\btau\cdot\bn)\in\Hm00(\gamma)$ is defined by
\beq\label{inparts}
\langle\tr^{d}_{\gamma}(\btau\cdot\bn),\tr^{g}_{\gamma}(v)\rangle_{\Hm00(\gamma)\times\H00(\gamma)}
:=(\btau,\nabla v)_{0,\O}+(\gradt\btau,v)_{0,\O}
\qquad\forall v\in H^{1}_{\p\O\setminus\gamma}(\O),
\eeq
whose right-hand side depends only on $\tr^{g}_{\gamma}(v)$ because it vanishes
on $H^{1}_{0}(\O)$.  Taking the infimum in \eqref{def_norm00} gives
\beq\label{eq:hdiv_trace}
\|\tr^{d}_{\gamma}(\btau\cdot\bn)\|_{\Hm00(\gamma)}\le\|\btau\|_{H(\divvr;\O)}.
\eeq
The quotient norms are equivalent to the intrinsic Sobolev--Slobodeckij norms
on $\p\O$, $\G_D$ and $\G_N$ \cite{Gri:11,LM:72}; we use them because they
give the constant one in \eqref{eq:quotient_lift}, \eqref{bd_on_12} and
\eqref{eq:hdiv_trace}.  When a trace is square integrable on a face $F$, we
write $\|v\|_{0,F}$ and $\|\btau\cdot\bn\|_{0,F}$ for the $L^{2}(F)$ norms of
$\tr^{g}_{F}(v)$ and $\tr^{d}_{F}(\btau\cdot\bn)$.

\subsection{Meshes and finite element spaces}
Let $\cT_h$ be a conforming, shape-regular, affine simplicial mesh of $\O$ such
that the relative interior of every boundary face lies in $\G_D$ or in $\G_N$,
and let $\cE_D$ and $\cE_N$ be the corresponding sets of boundary faces.  For a
boundary face $F$, $K_F$ is the element containing $F$.  We write $h_K$ and
$h_F$ for the diameters of $K$ and $F$ and $h_{\max}:=\max_{K\in\cT_h}h_K$;
elements and faces sharing a point have comparable diameters.  We normalize
$\operatorname{diam}\O\le1$, so $h_F\le h_{\max}\le1$.  For non-integer $s$,
$|\cdot|_{s,D}$ is the Sobolev--Slobodeckij seminorm on a domain or surface
$D$.  We write $X\lesssim Y$ if $X\le CY$ with $C$ independent of the mesh
size, and $X\simeq Y$ if $X\lesssim Y\lesssim X$.

Let $P_m(K)$ and $P_m(F)$ be the polynomials of degree at most $m$ on $K$ and
on a face $F$, and, for $k\ge0$, $\bRT_k(K):=P_k(K)^{d}+\bx\,P_k(K)$, whose
normal components on the faces of $K$ lie in $P_k(F)$.  The flux space
$\Sigma_h\subset H(\divvr;\O)$ and the scalar space $V_h\subset H^{1}(\O)$
consist of the functions whose restrictions to every $K\in\cT_h$ lie in
$\bRT_k(K)$ and in $P_{k+1}(K)$, respectively, and
$W^{D}_{h}:=\tr^{g}_{\G_D}(V_h)\subset H^{1/2}(\G_D)$ is the discrete Dirichlet
trace space.

\section{The first-order system: well-posedness and stability}
\label{sec:bvp}

Consider the second-order elliptic equation with mixed boundary conditions
\beq\label{pde}
\begin{gathered}
-\gradt(A\nabla u+\bb u)+\bbeta\cdot\nabla u+\mu u=f_1+\gradt\bff_2\quad\text{in }\O,\\
u=\phi\quad\text{on }\G_D,\qquad
-(A\nabla u+\bb u+\bff_2)\cdot\bn=\psi\quad\text{on }\G_N,
\end{gathered}
\eeq
with $f_1\in L^{2}(\O)$ and $\bff_2\in L^{2}(\O)^{d}$, so that the right-hand
side $f_1+\gradt\bff_2$ is in general only in $H^{-1}(\O)$.  With the flux
$\bsigma:=-(A\nabla u+\bb u+\bff_2)$, problem \eqref{pde} becomes the
first-order system
\beq\label{fos}
\begin{gathered}
\gradt\bsigma+\bbeta\cdot\nabla u+\mu u=f_1,\qquad
\bsigma+A\nabla u+\bb u=-\bff_2\quad\text{in }\O,\\
u=\phi\quad\text{on }\G_D,\qquad
\bsigma\cdot\bn=\psi\quad\text{on }\G_N.
\end{gathered}
\eeq
The Dirichlet condition is understood as $\tr^{g}_{\G_D}(u)=\phi$.  For
$u\in H^{1}(\O)$ and $\bff_2\in L^{2}(\O)^{d}$ we have
$\bsigma=-(A\nabla u+\bb u+\bff_2)\in L^{2}(\O)^{d}$ and
$\gradt\bsigma=f_1-\bbeta\cdot\nabla u-\mu u\in L^{2}(\O)$, so that
$\bsigma\in H(\divvr;\O)$, and the Neumann condition is understood as
$\tr^{d}_{\G_N}(\bsigma\cdot\bn)=\psi$ in $\Hm00(\G_N)$, through \eqref{inparts}
with $\gamma=\G_N$.  Since $\p\O\setminus\G_N=\overline{\G_D}$ up to a set of
measure zero, the test space in \eqref{inparts} is $H^{1}_{D}(\O)$.

Testing the first equation of \eqref{fos} with $v_0\in H^{1}_{D}(\O)$, using
\eqref{inparts} and $-\bsigma=A\nabla u+\bb u+\bff_2$, yields the weak
problem: find $u\in H^{1}(\O)$ with $\tr^{g}_{\G_D}(u)=\phi$ such that
\beq\label{weak}
a(u,v_0)=\ell(v_0)\qquad\forall v_0\in H^{1}_{D}(\O),
\eeq
where the bilinear form is $a(w,v):=(A\nabla w+\bb w,\nabla v)_{0,\O}
+(\bbeta\cdot\nabla w+\mu w,v)_{0,\O}$ and the linear form is
$\ell(v_0):=(f_1,v_0)_{0,\O}-(\bff_2,\nabla v_0)_{0,\O}
-\langle\psi,\tr^{g}_{\G_N}(v_0)\rangle_{\Hm00(\G_N)\times\H00(\G_N)}$.
By the Cauchy--Schwarz inequality and \eqref{bd_on_12},
\beq\label{continuity_a}
|a(w,v)|\le C_{con}\|w\|_{1,\O}\|v\|_{1,\O}\quad\forall w,v\in H^{1}(\O),
\qquad
|\ell(v_0)|\le\big(\|f_1\|_{0,\O}+\|\bff_2\|_{0,\O}
+\|\psi\|_{\Hm00(\G_N)}\big)\|v_0\|_{1,\O},
\eeq
with $C_{con}$ depending only on $\Lambda_1$, $\|\bb\|_{L^{\infty}(\O)}$,
$\|\bbeta\|_{L^{\infty}(\O)}$ and $\|\mu\|_{L^{\infty}(\O)}$.

\begin{assum}[uniqueness]\label{assum:infsup}
If $w_0\in H^{1}_{D}(\O)$ satisfies $a(w_0,v_0)=0$ for all
$v_0\in H^{1}_{D}(\O)$, then $w_0=0$.
\end{assum}

The terms of $a$ involving $\bb$, $\bbeta$ and $\mu$ are compact perturbations
of the principal part $(A\nabla w,\nabla v)_{0,\O}$, which is coercive on
$H^{1}_{D}(\O)$ because $\G_D$ has positive measure.  The operator
$H^{1}_{D}(\O)\to(H^{1}_{D}(\O))'$ induced by $a$ is therefore Fredholm of
index zero, and Assumption~\ref{assum:infsup} makes it injective, hence
bijective.  Consequently there is a constant $\beta_a>0$ such that
\beq\label{infsup}
\inf_{0\neq w_0\in H^{1}_{D}(\O)}\ \sup_{0\neq v_0\in H^{1}_{D}(\O)}
\frac{a(w_0,v_0)}{\|w_0\|_{1,\O}\|v_0\|_{1,\O}}\ \geq\ \beta_a ,
\eeq
and $a(w_0,v_0)=0$ for all $w_0\in H^{1}_{D}(\O)$ implies $v_0=0$.  These are
the two hypotheses of the Ne\v{c}as--Babu\v{s}ka theorem; both follow from the
single uniqueness assumption.  This minimal assumption, and its role in the
coercivity of first-order system least-squares methods for general
second-order elliptic equations, are discussed in detail
in \cite{Zhang:23}.

\begin{thm}[well-posedness and stability]\label{thm:stability}
Let the uniqueness assumption, Assumption~\ref{assum:infsup}, hold.  For every $f_1\in L^{2}(\O)$,
$\bff_2\in L^{2}(\O)^{d}$, $\phi\in H^{1/2}(\G_D)$ and $\psi\in\Hm00(\G_N)$,
problem \eqref{weak} has a unique solution $u$.  With
$\bsigma:=-(A\nabla u+\bb u+\bff_2)$, the pair $(\bsigma,u)$ is the unique
solution of \eqref{fos} in $H(\divvr;\O)\times H^{1}(\O)$, and
\beq\label{eq:stability}
\|u\|_{1,\O}+\|\bsigma\|_{H(\divvr;\O)}\ \le\
C_{S}\big(\|f_1\|_{0,\O}+\|\bff_2\|_{0,\O}+\|\psi\|_{\Hm00(\G_N)}
+\|\phi\|_{1/2,\G_D}\big).
\eeq
\end{thm}

\begin{proof}
Let $u_\phi\in H^{1}(\O)$ with $\tr^{g}_{\G_D}(u_\phi)=\phi$.  By
\eqref{continuity_a}, \eqref{infsup} and the Ne\v{c}as--Babu\v{s}ka theorem
there is a unique $u_0\in H^{1}_{D}(\O)$ with
$a(u_0,v_0)=\ell(v_0)-a(u_\phi,v_0)$ for all $v_0\in H^{1}_{D}(\O)$, and
$\beta_a\|u_0\|_{1,\O}\le\|f_1\|_{0,\O}+\|\bff_2\|_{0,\O}
+\|\psi\|_{\Hm00(\G_N)}+C_{con}\|u_\phi\|_{1,\O}$.  Thus $u:=u_0+u_\phi$ is the
unique solution of \eqref{weak}, and the infimum over $u_\phi$ and
\eqref{eq:quotient_inf} bound $\|u\|_{1,\O}$ as in \eqref{eq:stability}.
Testing \eqref{weak} with $v_0\in C_c^\infty(\O)$ gives
$\gradt\bsigma=f_1-\bbeta\cdot\nabla u-\mu u$, and \eqref{inparts} with
$\gamma=\G_N$ gives $\tr^{d}_{\G_N}(\bsigma\cdot\bn)=\psi$, so $(\bsigma,u)$
solves \eqref{fos}; the bounds for $\|\bsigma\|_{0,\O}$ and
$\|\gradt\bsigma\|_{0,\O}$ follow from the two equations in $\O$.
\end{proof}

\begin{cor}[residual representation]\label{lem:residual_rep}
Let the uniqueness assumption, Assumption~\ref{assum:infsup}, hold, let
$(\btau,v)\in H(\divvr;\O)\times H^{1}(\O)$, and set
$$
R_1:=A^{-\half}(\btau+A\nabla v+\bb v),\quad
R_2:=\gradt\btau+\bbeta\cdot\nabla v+\mu v.
$$
Then $(\btau,v)$ is the unique solution of \eqref{fos} with data $f_1=R_2$,
$\bff_2=-A^{\half}R_1$, $\phi=\tr^{g}_{\G_D}(v)$ and $\psi=\tr^{d}_{\G_N}(\btau\cdot\bn)$, and
\beq\label{eq:residual_stability}
\|v\|_{1,\O}+\|\btau\|_{H(\divvr;\O)}
\ \le\ C_{S}\Big(\big(1+\Lambda_1^{\half}\big)\big(\|R_1\|_{0,\O}+\|R_2\|_{0,\O}\big)
+\|\tr^{g}_{\G_D}(v)\|_{1/2,\G_D}+\|\tr^{d}_{\G_N}(\btau\cdot\bn)\|_{\Hm00(\G_N)}\Big).
\eeq
\end{cor}

\begin{proof}
The data are admissible: $\tr^{g}_{\G_D}(v)\in H^{1/2}(\G_D)$ by
\eqref{eq:quotient_lift} and $\tr^{d}_{\G_N}(\btau\cdot\bn)\in\Hm00(\G_N)$ by
\eqref{eq:hdiv_trace}.  Substitution shows that
$(\btau,v)$ solves \eqref{fos} with these data, and
Theorem~\ref{thm:stability} gives \eqref{eq:residual_stability}.
\end{proof}

\section{The least-squares functional with weakly imposed boundary conditions: norm equivalence and error equivalence}
\label{sec:functional}

\begin{assum}[regular boundary data]\label{assum:data}
The Dirichlet datum satisfies $\phi\in H^{1/2}(\G_D)$, and the Neumann datum
satisfies $\psi\in L^{2}(\G_N)$.
\end{assum}

The Dirichlet assumption is the natural one: $H^{1/2}(\G_D)$ is the trace space
of $H^{1}(\O)$ on $\G_D$.  The Neumann assumption is stronger than the
$\psi\in\Hm00(\G_N)$ allowed in Theorem~\ref{thm:stability}, because the
Neumann residual will be penalized in $L^{2}(\G_N)$ and the penalty must be
finite.

Two ambient spaces are used.  The flux space
\beq\label{def:Sstar}
\Sigma:=\{\btau\in H(\divvr;\O):\ \tr^{d}_{\G_N}(\btau\cdot\bn)\in L^{2}(\G_N)\}
\eeq
is the set on which the Neumann penalty is finite, and $\Sigma_h\subset\Sigma$.
The scalar space
\beq\label{def:Vh}
V^{D}_{\ast}:=\{v\in H^{1}(\O):\ \tr^{g}_{\G_D}(v)\in W^{D}_{h}\}
\eeq
restricts only the Dirichlet trace, so that an inverse inequality on discrete
traces bounds the $H^{1/2}(\G_D)$ norm by the Dirichlet penalty.  Equivalently,
$V^{D}_{\ast}=V_h+H^{1}_{D}(\O)$: if $\tr^{g}_{\G_D}(v)=\tr^{g}_{\G_D}(w_h)$
with $w_h\in V_h$, then $v-w_h\in H^{1}_{D}(\O)$.  In particular,
$\Sigma_h\times V_h\subset\Sigma\times V^{D}_{\ast}$.

For $(\btau,v)\in\Sigma\times H^{1}(\O)$ define
\begin{align}
\label{def:J}
J(\btau,v;f_1,\bff_2,\phi,\psi)
&:=\big\|A^{-\half}(\btau+A\nabla v+\bb v+\bff_2)\big\|_{0,\O}^{2}
+\big\|\gradt\btau+\bbeta\cdot\nabla v+\mu v-f_1\big\|_{0,\O}^{2}\\
\nonumber
&\qquad
+\sum_{F\in\cE_D}h_F^{-1}\|v-\phi\|_{0,F}^{2}
+\|\btau\cdot\bn-\psi\|_{0,\G_N}^{2},
\end{align}
and write $J(\btau,v;0):=J(\btau,v;0,\bzero,0,0)$.  The associated
mesh-dependent norm is
\beq\label{def:norm}
\tri(\btau,v)\tri^{2}
:=\|v\|_{1,\O}^{2}+\|\btau\|_{H(\divvr;\O)}^{2}
+\sum_{F\in\cE_D}h_F^{-1}\|v\|_{0,F}^{2}
+\|\btau\cdot\bn\|_{0,\G_N}^{2}.
\eeq

The right-hand side $f_1+\gradt\bff_2$ of \eqref{pde} is a general element of
$H^{-1}(\O)$, written out explicitly: every element of $H^{-1}(\O)$ has this
form with $f_1\in L^{2}(\O)$ and $\bff_2\in L^{2}(\O)^{d}$.  Such data are
usually treated as abstract elements of a dual space and measured in negative
norms, for instance through the discrete $H^{-1}$ inner product of
\cite{BLP:97}.  With the explicit form, the functional \eqref{def:J} consists
of $L^{2}$ norms only: $\bff_2$ is absorbed into the flux
$\bsigma=-(A\nabla u+\bb u+\bff_2)$, an exact change of variables, so that
$\gradt\bsigma\in L^{2}(\O)$ and the divergence residual is measured in
$L^{2}(\O)$.  The negative norm belongs to the data, not to the residual.  This
treatment of $H^{-1}$ data was introduced in \cite{Zhang:23}, and
\cite[Sect.~2]{WuZhang:26} makes it concrete for linear functionals on the
boundary, with quantities of interest as examples.

\begin{lem}[continuity]\label{lem:continuity}
There is a constant $C_2$ depending only on $\Lambda_0$, $\Lambda_1$,
$\|\bb\|_{L^{\infty}(\O)}$, $\|\bbeta\|_{L^{\infty}(\O)}$ and
$\|\mu\|_{L^{\infty}(\O)}$ with $J(\btau,v;0)\le C_2\,\tri(\btau,v)\tri^{2}$ for all $(\btau,v)\in\Sigma\times H^{1}(\O)$.
\end{lem}

\begin{proof}
The two volume terms are bounded by the triangle inequality and the bounds on
the coefficients, and the boundary terms of $J(\cdot;0)$ and $\tri\cdot\tri$
coincide.
\end{proof}

The two boundary inequalities used below are of different nature.  On $\G_N$ the
relevant inequality is the embedding $L^{2}(\G_N)\hookrightarrow\Hm00(\G_N)$ of the Gelfand triple,
\beq\label{eq:embedN}
\|q\|_{\Hm00(\G_N)}\le C\|q\|_{0,\G_N}\qquad\forall q\in L^{2}(\G_N).
\eeq
On $\G_D$ it is an inverse inequality on the discrete trace space.

\begin{lem}[discrete Dirichlet trace inequality]\label{lem:inverseD}
There is $C$, depending only on shape regularity, $k$ and $d$, such that
\beq\label{eq:inverseD}
\|g\|_{1/2,\G_D}\le C\Big(\sum_{F\in\cE_D}h_F^{-1}\|g\|_{0,F}^{2}\Big)^{1/2}
\qquad\forall g\in W^{D}_{h}.
\eeq
\end{lem}

\begin{proof}
Let $L_hg\in V_h$ take the Lagrange nodal values of $g$ at the nodes on
$\overline{\G_D}$ and the value zero at all other nodes.  Then
$\tr^{g}_{\G_D}(L_hg)=g$.  For a node $z$ on
a face $F\in\cE_D$, $|g(z)|^{2}\le Ch_F^{1-d}\|g\|^{2}_{0,F}$, and on every
element $K$, $\|L_hg\|^{2}_{0,K}+h_K^{2}\|\nabla L_hg\|^{2}_{0,K}\le
Ch_K^{d}\sum_z|g(z)|^{2}$, the sum running over the nodes of $K$ on
$\overline{\G_D}$.  Elements and faces sharing a node have comparable diameters
and the node patches overlap finitely, so
$\|L_hg\|^{2}_{1,\O}\le C\sum_{F\in\cE_D}(h_F+h_F^{-1})\|g\|^{2}_{0,F}$, and
since $h_F\le1$ the right-hand side is at most
$2C\sum_{F\in\cE_D}h_F^{-1}\|g\|^{2}_{0,F}$.  Finally,
$\|g\|_{1/2,\G_D}\le\|L_hg\|_{1,\O}$ by \eqref{eq:quotient_lift}.
\end{proof}

\begin{thm}[norm equivalence]\label{thm:threshold}
Let the uniqueness assumption, Assumption~\ref{assum:infsup}, hold.
Then
\beq\label{eq:normequiv}
C_1\tri(\btau,v)\tri^{2}\ \le\ J(\btau,v;0)\ \le\ C_2\tri(\btau,v)\tri^{2}
\qquad\forall(\btau,v)\in\Sigma\times V^{D}_{\ast},
\eeq
with $C_1,C_2>0$ independent of the mesh size.
\end{thm}

\begin{proof}
The upper bound is Lemma~\ref{lem:continuity}.  For the lower bound, the
functional is $J(\btau,v;0)=\|R_1\|_{0,\O}^{2}+\|R_2\|_{0,\O}^{2}
+\sum_{F\in\cE_D}h_F^{-1}\|v\|_{0,F}^{2}+\|\btau\cdot\bn\|_{0,\G_N}^{2}$, with
$R_1,R_2$ as in Corollary~\ref{lem:residual_rep}.  Since
$\tr^{g}_{\G_D}(v)\in W^{D}_{h}$, \eqref{eq:inverseD} and \eqref{eq:embedN}
bound the two boundary norms in \eqref{eq:residual_stability} by
$CJ(\btau,v;0)^{1/2}$, so $\|v\|_{1,\O}+\|\btau\|_{H(\divvr;\O)}\le
CJ(\btau,v;0)^{1/2}$.  The boundary terms of $\tri(\btau,v)\tri^{2}$ are
terms of $J(\btau,v;0)$.
\end{proof}

\begin{rem}[where discreteness enters]\label{rem:where_discrete}
The only discrete ingredient in Theorem~\ref{thm:threshold} is the space $V^{D}_{\ast}$ of \eqref{def:Vh}, whose
Dirichlet trace is restricted to $W^{D}_{h}$.  That restriction serves a single
purpose: to apply the inverse inequality \eqref{eq:inverseD}, which bounds the
fractional norm $\|\tr^{g}_{\G_D}(v)\|_{1/2,\G_D}$ by the weighted $L^{2}$
penalty and holds on the finite-dimensional trace space.  On $\G_N$ the
corresponding step is the embedding \eqref{eq:embedN}, valid for every
$q\in L^{2}(\G_N)$, so neither the flux nor the interior values of the scalar
are restricted.  The functional itself is defined on $\Sigma\times H^{1}(\O)$; the
restriction to $V^{D}_{\ast}$ belongs to the proof of the lower bound.
\end{rem}

The same ingredients give a second equivalence on $\Sigma\times V^{D}_{\ast}$,
the error equivalence, with respect to the exact solution.  The trace of the exact
solution on $\G_D$ is the datum $\phi$, which in general lies outside
$W^{D}_{h}$, so the Dirichlet datum enters through a data term.  It is
measured with a projection $Q_h$ of $L^{2}(\G_D)$ onto $W^{D}_{h}$, so that
$Q_hw_h=w_h$ for $w_h\in W^{D}_{h}$, bounded in the $h_F^{-1}$-weighted
$L^{2}(\G_D)$ norm:
\beq\label{def:Qh}
\sum_{F\in\cE_D}h_F^{-1}\|Q_hg\|^{2}_{0,F}
\le C_Q^{2}\sum_{F\in\cE_D}h_F^{-1}\|g\|^{2}_{0,F}
\qquad\forall g\in L^{2}(\G_D).
\eeq
Such projections exist.  The $h_F^{-1}$-weighted $L^{2}(\G_D)$ projection $P_h$,
\beq\label{def:Ph}
\sum_{F\in\cE_D}h_F^{-1}(g-P_hg,\,w_h)_{0,F}=0
\qquad\forall w_h\in W^{D}_{h},
\eeq
satisfies \eqref{def:Qh} with $C_Q=1$.  The Scott--Zhang projection
$S_h:L^{2}(\G_D)\to W^{D}_{h}$ \cite{SZ:90}, constructed on the
$(d-1)$-dimensional mesh $\cE_D$, satisfies \eqref{def:Qh} by its local
$L^{2}$-stability \cite[eq.~(3.7)]{AFKPP:13} and the comparability of
neighboring face diameters; moreover, $I-S_h$ is bounded in $H^{1/2}(\G_D)$
\cite[eq.~(3.10)]{AFKPP:13}.

\begin{thm}[error equivalence]\label{thm:aposteriori}
Let the uniqueness assumption, Assumption~\ref{assum:infsup}, hold, let the
data satisfy Assumption~\ref{assum:data}, let $(\bsigma,u)$ be the solution of
\eqref{fos}, and let $Q_h$ satisfy \eqref{def:Qh}.
Then, for all $(\btau,v)\in\Sigma\times V^{D}_{\ast}$,
\begin{align}
\label{eq:reliability}
c\,[J(\btau,v;f_1,\bff_2,\phi,\psi)^{1/2}+\|\phi-Q_h\phi\|_{1/2,\G_D}]
&\le\tri(\bsigma-\btau,u-v)\tri\\
\nonumber
&\le C\,[J(\btau,v;f_1,\bff_2,\phi,\psi)^{1/2}+\|\phi-Q_h\phi\|_{1/2,\G_D}],
\end{align}
with constants $c,C>0$ depending on $C_Q$ but independent of the mesh size.
\end{thm}

\begin{proof}
Let $e:=(\bsigma-\btau,u-v)$ and $\eta:=J(\btau,v;f_1,\bff_2,\phi,\psi)^{1/2}$.
Since $(\bsigma,u)$ solves \eqref{fos} and the residuals are affine,
\beq\label{eq:eta_identity}
\eta^{2}=J(\btau-\bsigma,v-u;0),
\eeq
so $\eta\le C_2^{1/2}\tri e\tri$ by Lemma~\ref{lem:continuity}.  Since
$\tr^{g}_{\G_D}(v)\in W^{D}_{h}$, $Q_h\phi-v=Q_h(\phi-v)\in W^{D}_{h}$, and
\eqref{eq:inverseD} and \eqref{def:Qh} give
\beq\label{eq:Qh_discrete}
\|Q_h\phi-v\|_{1/2,\G_D}\le C\Big(\sum_{F\in\cE_D}h_F^{-1}\|\phi-v\|^{2}_{0,F}\Big)^{1/2}.
\eeq
For the upper bound we apply \eqref{eq:residual_stability} to $e$: by
\eqref{eq:eta_identity} its volume residuals are at most $\sqrt2\,\eta$, by
\eqref{eq:embedN} its Neumann trace $\psi-\btau\cdot\bn$ is at most $C\eta$ in
$\Hm00(\G_N)$, and by \eqref{eq:Qh_discrete} its Dirichlet trace $\phi-v$ is at
most $\|\phi-Q_h\phi\|_{1/2,\G_D}+C\eta$ in $H^{1/2}(\G_D)$; the boundary terms
of $\tri e\tri^{2}$ are terms of $\eta^{2}$.  For the lower bound,
\eqref{eq:quotient_lift} and \eqref{eq:Qh_discrete} give
$\|\phi-Q_h\phi\|_{1/2,\G_D}\le\|\tr^{g}_{\G_D}(u-v)\|_{1/2,\G_D}+\|Q_h\phi-v\|_{1/2,\G_D}\le
C\tri e\tri$.
\end{proof}

Among all projections onto $W^{D}_{h}$, the $H^{1/2}(\G_D)$-orthogonal
projection gives the smallest data term, the best approximation error of
$\phi$ in $H^{1/2}(\G_D)$ by $W^{D}_{h}$.  As for the $H^{1/2}$-stable
projections of \cite{AFKPP:13}, every projection satisfying \eqref{def:Qh}
whose data term is equivalent to it gives the same result.  The projection
$S_h$ is one: since $S_hw_h=w_h$,
$\|\phi-S_h\phi\|_{1/2,\G_D}=\|(I-S_h)(\phi-w_h)\|_{1/2,\G_D}$ for every
$w_h\in W^{D}_{h}$, and the boundedness of $I-S_h$ in $H^{1/2}(\G_D)$ gives
\beq\label{eq:SZ_quasiopt}
\|\phi-S_h\phi\|_{1/2,\G_D}\le C\inf_{w_h\in W^{D}_{h}}\|\phi-w_h\|_{1/2,\G_D}
=:C\,\delta_h(\phi).
\eeq

\begin{cor}[best approximation of the Dirichlet datum]\label{cor:best_approx}
Let the uniqueness assumption, Assumption~\ref{assum:infsup}, hold, let the
data satisfy Assumption~\ref{assum:data}, and let $(\bsigma,u)$ be the solution
of \eqref{fos}.  Then, for all $(\btau,v)\in\Sigma\times V^{D}_{\ast}$,
$$
c\,[J(\btau,v;f_1,\bff_2,\phi,\psi)^{1/2}+\delta_h(\phi)]
\le\tri(\bsigma-\btau,u-v)\tri
\le C\,[J(\btau,v;f_1,\bff_2,\phi,\psi)^{1/2}+\delta_h(\phi)].
$$
\end{cor}

\begin{proof}
The upper bound follows from Theorem~\ref{thm:aposteriori} with $Q_h=S_h$ and
\eqref{eq:SZ_quasiopt}.  The lower bound follows from \eqref{eq:eta_identity}, Lemma~\ref{lem:continuity} and $\delta_h(\phi)\le\|\phi-v\|_{1/2,\G_D}\le\|u-v\|_{1,\O}\le\tri(\bsigma-\btau,u-v)\tri$,
by \eqref{eq:quotient_lift}.
\end{proof}

\section{LSFEM with weakly imposed boundary conditions and a priori error estimates}
\label{sec:apriori}

The least-squares finite element method with weakly imposed boundary
conditions reads: find $(\bsigma_h,u_h)\in\Sigma_h\times V_h$ such that
\beq\label{eq:ls_min}
J(\bsigma_h,u_h;f_1,\bff_2,\phi,\psi)
=\min_{(\btau_h,v_h)\in\Sigma_h\times V_h}J(\btau_h,v_h;f_1,\bff_2,\phi,\psi),
\eeq
with the functional \eqref{def:J}.  The exact data $\phi$ and $\psi$ enter the functional directly.

The minimization problem \eqref{eq:ls_min} is equivalent to its normal
equations, whose bilinear form $b$ satisfies
$b\big((\btau,v),(\btau,v)\big)=J(\btau,v;0)$.  Since
$\Sigma_h\times V_h\subset\Sigma\times V^{D}_{\ast}$, Theorem~\ref{thm:threshold}
shows that $b$ is coercive on $\Sigma_h\times V_h$ with constant $C_1$, so
\eqref{eq:ls_min} has a unique solution, and the matrix of the normal equations
is symmetric positive definite.

Three operators are used in the analysis.  For a face $F$, $\Pi^{F}_{k}$
denotes the $L^{2}(F)$-orthogonal projection onto $P_k(F)$, and for
$\psi\in L^{2}(\G_N)$ we write $\Pi_k\psi$ for the face-wise projection,
$(\Pi_k\psi)|_F:=\Pi^{F}_{k}(\psi|_F)$ for $F\in\cE_N$.  The canonical
Raviart--Thomas interpolation operator $\Pi^{RT}_h$ onto $\Sigma_h$ is defined
on $H(\divvr;\O)\cap H^{r}(\O)^{d}$ for any $r>0$ \cite{LiangZhang:23}, by the face moments
$\int_F(\Pi^{RT}_h\btau)\cdot\bn\,q\,ds=\int_F\btau\cdot\bn\,q\,ds$ for all
$q\in P_k(F)$ and all faces $F$, together with the interior moments against
$P_{k-1}(K)^{d}$ when $k\ge1$.  In particular
\beq\label{eq:RT_moment}
(\Pi^{RT}_h\btau)\cdot\bn|_F=\Pi^{F}_{k}(\btau\cdot\bn|_F)
\eeq
on every face on which $\btau\cdot\bn$ is square integrable.  For $r\le\frac12$
the face moments are understood in the generalized normal-trace sense; they
coincide with the face integrals above whenever the normal
trace lies in $L^{2}(F)$, in particular on $\G_N$ for the exact flux, whose
normal trace is $\psi\in L^{2}(\G_N)$.  For
$0<s\le k+1$, $\btau\in H^{s}(\O)^{d}$ and $\gradt\btau\in H^{s}(\O)$,
\beq\label{eq:RT_approx}
\|\btau-\Pi^{RT}_h\btau\|_{H(\divvr;\O)}
\le Ch_{\max}^{s}\big(|\btau|_{s,\O}+|\gradt\btau|_{s,\O}\big);
\eeq
for non-integer $s$ this follows from averaged Taylor polynomials and the Piola
transformation, as in \cite[(2.5)--(2.6)]{LiangZhang:23}.  Finally, $I_h$
denotes any quasi-interpolation operator onto $V_h$ that is defined on
$H^{1}(\O)$ and satisfies the standard local estimates
$\|u-I_hu\|_{0,K}+h_K|u-I_hu|_{1,K}\le Ch_K^{1+s}|u|_{1+s,\o_K}$ for
$0<s\le k+1$, where $\o_K$ is the patch of elements touching $K$; the
Scott--Zhang operator \cite{SZ:90} is one choice.  Then
\beq\label{eq:SZ_approx}
\|u-I_hu\|_{1,\O}\le Ch_{\max}^{s}|u|_{1+s,\O},
\qquad
\Big(\sum_{F\in\cE_D}h_F^{-1}\|u-I_hu\|^{2}_{0,F}\Big)^{1/2}\le Ch_{\max}^{s}|u|_{1+s,\O}.
\eeq
Only the second bound is specific to the weak Dirichlet condition, and it
involves only the elements touching $\G_D$: it follows from the local estimates
on $K_F$, the element-wise trace inequality
$\|w\|^{2}_{0,F}\le C\big(h_F^{-1}\|w\|^{2}_{0,K_F}+h_F|w|^{2}_{1,K_F}\big)$ and
the finite overlap of the patches.

\begin{lem}[Pythagoras on $\G_N$]\label{lem:pythagoras}
Let $\psi_h:=\Pi_k\psi$.  For every $\btau_h\in\Sigma_h$,
\beq\label{eq:pythagoras}
\|\btau_h\cdot\bn-\psi\|^{2}_{0,\G_N}
=\|\btau_h\cdot\bn-\psi_h\|^{2}_{0,\G_N}+\|\psi-\psi_h\|^{2}_{0,\G_N},
\eeq
so that, over $\Sigma_h\times V_h$, $J(\btau_h,v_h;f_1,\bff_2,\phi,\psi) =J(\btau_h,v_h;f_1,\bff_2,\phi,\psi_h) +\|\psi-\psi_h\|^{2}_{0,\G_N}$, the last term being independent of $(\btau_h,v_h)$.  In particular the two
minimization problems have the same solution.
\end{lem}

\begin{proof}
On each face $F\in\cE_N$ we write
$\btau_h\cdot\bn-\psi=(\btau_h\cdot\bn-\psi_h)+(\psi_h-\psi)$.  The first summand
belongs to $P_k(F)$ because $\btau_h\cdot\bn|_F\in P_k(F)$, and the second is
$L^{2}(F)$-orthogonal to $P_k(F)$, so the cross term vanishes.  Summing over
the faces gives \eqref{eq:pythagoras}.
\end{proof}

\begin{thm}[a priori estimate]\label{thm:apriori}
Let the uniqueness assumption, Assumption~\ref{assum:infsup}, hold, let the
data satisfy Assumption~\ref{assum:data}, and let $(\bsigma,u)$ be the solution
of \eqref{fos} and $(\bsigma_h,u_h)$ the solution of \eqref{eq:ls_min}.  Let
$0<s\le k+1$, and assume $u\in H^{1+s}(\O)$, $\bsigma\in H^{s}(\O)^{d}$ and
$\gradt\bsigma\in H^{s}(\O)$.  Then
\beq\label{eq:apriori}
\|u-u_h\|_{1,\O}+\|\bsigma-\bsigma_h\|_{H(\divvr;\O)}
\le C\,h_{\max}^{s}\big(|u|_{1+s,\O}+|\bsigma|_{s,\O}+|\gradt\bsigma|_{s,\O}\big),
\eeq
with $C$ independent of the mesh size.
\end{thm}

\begin{proof}
Let $z_h:=(\bsigma_h,u_h)$, $z_h^{\ast}:=(\Pi^{RT}_h\bsigma,I_hu)\in\Sigma_h\times V_h$
and $J(\cdot):=J(\cdot;f_1,\bff_2,\phi,\Pi_k\psi)$.  By
Lemma~\ref{lem:pythagoras}, $z_h$ minimizes $J$ over $\Sigma_h\times V_h$, so
$J(z_h-z_h^{\ast};0)^{1/2}\le J(z_h)^{1/2}+J(z_h^{\ast})^{1/2}\le2J(z_h^{\ast})^{1/2}$,
and Theorem~\ref{thm:threshold} gives
\beq\label{eq:apriori_step1}
\|u_h-I_hu\|_{1,\O}+\|\bsigma_h-\Pi^{RT}_h\bsigma\|_{H(\divvr;\O)}
\le2\sqrt2\,C_1^{-1/2}J(z_h^{\ast})^{1/2}.
\eeq
Since $(\bsigma,u)$ solves \eqref{fos} and $\Pi^{RT}_h\bsigma\cdot\bn=\Pi_k\psi$
on $\G_N$ by \eqref{eq:RT_moment}, the Neumann term of $J(z_h^{\ast})$
vanishes, its Dirichlet term is $\sum_{F\in\cE_D}h_F^{-1}\|u-I_hu\|^{2}_{0,F}$,
and its volume terms are at most
$C\big(\|\bsigma-\Pi^{RT}_h\bsigma\|_{H(\divvr;\O)}+\|u-I_hu\|_{1,\O}\big)^{2}$.
By \eqref{eq:RT_approx} and \eqref{eq:SZ_approx},
\beq\label{eq:apriori_step2}
J(z_h^{\ast})^{1/2}
\le Ch_{\max}^{s}\big(|u|_{1+s,\O}+|\bsigma|_{s,\O}+|\gradt\bsigma|_{s,\O}\big),
\eeq
and the triangle inequality with \eqref{eq:RT_approx}, \eqref{eq:SZ_approx},
\eqref{eq:apriori_step1} and \eqref{eq:apriori_step2} gives \eqref{eq:apriori}.
\end{proof}

\begin{rem}[on the a priori estimate]\label{rem:zero_defect}
The a priori estimate is stated in $H(\divvr;\O)\times H^{1}(\O)$, the norm of
the problem, and not in $\tri\cdot\tri$.  The Dirichlet term of
$\tri(\bsigma-\bsigma_h,u-u_h)\tri^{2}$ is of the same order, by
\eqref{eq:SZ_approx}, \eqref{eq:apriori_step1} and \eqref{eq:apriori_step2},
but by Lemma~\ref{lem:pythagoras} the Neumann term is at least
$\|\psi-\Pi_k\psi\|^{2}_{0,\G_N}$, which for $s>\frac12$ is in general only of
order $h_{\max}^{2s-1}$; $\tri\cdot\tri$ is the norm of the a posteriori
estimate.  The unweighted $L^{2}(\G_N)$ penalty, although stronger than the
natural $\Hm00(\G_N)$ norm, enters the proof only through
$\|\Pi^{RT}_h\bsigma\cdot\bn-\Pi_k\psi\|_{0,\G_N}$, which vanishes, so
\eqref{eq:apriori} is optimal.
\end{rem}

\section{A posteriori error estimation for LSFEM with weakly imposed boundary conditions}
\label{sec:aposteriori}

The functional itself is the estimator.  Every term of \eqref{def:J} is a sum
of squared $L^{2}$ norms over elements or faces, so $\eta_h^{2}:=J(\bsigma_h,u_h;f_1,\bff_2,\phi,\psi)$ is a sum of element
contributions, each consisting of the two volume residuals on an element $K$
and the boundary residuals on those faces of $K$ that lie in $\p\O$.  The same minimization therefore
returns local error indicators, computable when the boundary data are regular in the sense of
Assumption~\ref{assum:data}.

Since $\Sigma_h\times V_h\subset\Sigma\times V^{D}_{\ast}$,
Theorem~\ref{thm:aposteriori} applies to the solution $(\bsigma_h,u_h)$ of
\eqref{eq:ls_min}: for every $Q_h$ satisfying \eqref{def:Qh},
\beq\label{eq:apost_h}
c\,[\eta_h+\|\phi-Q_h\phi\|_{1/2,\G_D}]
\le\tri(\bsigma-\bsigma_h,u-u_h)\tri
\le C\,[\eta_h+\|\phi-Q_h\phi\|_{1/2,\G_D}].
\eeq
The data term in \eqref{eq:reliability} depends only on $\phi$, the mesh and
$Q_h$, and by Corollary~\ref{cor:best_approx} it can be replaced by
$\delta_h(\phi)$.  The next lemma bounds it, together with the weighted
$L^{2}$ part of the Dirichlet error, by the error of any $w_h\in W^{D}_{h}$.

\begin{lem}[quasi-optimality of $Q_h$ on $\G_D$]\label{lem:oscD}
Let $Q_h$ satisfy \eqref{def:Qh}.  For every $\phi\in H^{1/2}(\G_D)$,
\beq\label{eq:oscD_bound}
\|\phi-Q_h\phi\|_{1/2,\G_D}
+\Big(\sum_{F\in\cE_D}h_F^{-1}\|\phi-Q_h\phi\|^{2}_{0,F}\Big)^{1/2}
\le C\inf_{w_h\in W^{D}_{h}}\Big\{\|\phi-w_h\|_{1/2,\G_D}
+\Big(\sum_{F\in\cE_D}h_F^{-1}\|\phi-w_h\|^{2}_{0,F}\Big)^{1/2}\Big\},
\eeq
with $C$ depending only on $C_Q$, shape regularity, $k$ and $d$.
\end{lem}

\begin{proof}
Let $w_h\in W^{D}_{h}$.  Since $Q_hw_h=w_h$,
$\phi-Q_h\phi=(\phi-w_h)-Q_h(\phi-w_h)$ with $Q_h(\phi-w_h)\in W^{D}_{h}$.  By
\eqref{def:Qh}, the weighted $L^{2}$ norm of $Q_h(\phi-w_h)$ is at most $C_Q$
times that of $\phi-w_h$, and by \eqref{eq:inverseD} and \eqref{def:Qh} its
$H^{1/2}(\G_D)$ norm is at most a constant times the same quantity.  The
triangle inequality gives \eqref{eq:oscD_bound}, since $w_h$ is arbitrary.
\end{proof}

To make the right-hand side of \eqref{eq:oscD_bound}
computable, it suffices to insert one $w_h\in W^{D}_{h}$ whose two error terms
admit a local bound.  One choice is $w_h=S_h\phi$; for $v\in H^{1}(\G_D)$,
\beq\label{eq:SZ_bdry}
\|v-S_hv\|_{0,F}\le Ch_F\|\nabla_{\G}v\|_{0,\o^{D}_F}\qquad\forall F\in\cE_D,
\eeq
where $\nabla_{\G}$ is the surface gradient and $\o^{D}_F$ is the patch of
faces of $\cE_D$ touching $F$.  For $\phi\in H^{1}(\G_D)$ we define the
Dirichlet data oscillation
\beq\label{def:oscD}
\osc_D(\phi):=\Big(\sum_{F\in\cE_D}h_F\|\nabla_{\G}(\phi-S_h\phi)\|^{2}_{0,F}\Big)^{1/2},
\eeq
a variant of the oscillations used in adaptive finite element methods
\cite{AFKPP:13,FPP:14}.

\begin{cor}\label{cor:oscD}
Let the assumptions of Theorem~\ref{thm:aposteriori} hold, and let
$(\bsigma_h,u_h)$ be the solution of \eqref{eq:ls_min}.  If
$\phi\in H^{1}(\G_D)$, then the left-hand side of \eqref{eq:oscD_bound} is
bounded by $C\,\osc_D(\phi)$, and
\beq\label{eq:reliability_osc}
\tri(\bsigma-\bsigma_h,u-u_h)\tri\le C\,[\eta_h+\osc_D(\phi)],
\eeq
in which both terms are sums of face-wise or element-wise contributions.
\end{cor}

\begin{proof}
Take $w_h=S_h\phi$ in Lemma~\ref{lem:oscD}.  Put $w:=\phi-S_h\phi$.  Since $S_h$
is a projection onto $W^{D}_{h}$, $S_hw=0$, and \eqref{eq:SZ_bdry} applied to
$w$ gives $\|w\|_{0,F}=\|w-S_hw\|_{0,F}\le Ch_F\|\nabla_{\G}w\|_{0,\o^{D}_F}$.
The patches overlap finitely and neighboring faces have comparable diameters,
so $\sum_{F\in\cE_D}h_F^{-1}\|w\|^{2}_{0,F}\le C\,\osc_D(\phi)^{2}$.  For the
fractional norm, the localization
$\|\phi-S_h\phi\|_{1/2,\G_D}\le C\,\osc_D(\phi)$ is
\cite[eq.~(3.11)]{AFKPP:13}, derived there from \cite[Thm.~3]{KOP:13}.  The
estimate \eqref{eq:reliability_osc} then follows from \eqref{eq:apost_h}.
\end{proof}

\begin{rem}[the Dirichlet data term]\label{rem:oscD}
The face-wise oscillation
$\big(\sum_{F\in\cE_D}h_F\|(1-\Pi^{F}_{k})\nabla_{\G}\phi\|^{2}_{0,F}\big)^{1/2}$
of \cite[eq.~(2.8)]{AFKPP:13} is bounded by $\osc_D(\phi)$, because
$\nabla_{\G}(S_h\phi)|_F\in P_k(F)^{d-1}$; the reverse inequality is
\cite[Prop.~3.1]{AFKPP:13} for meshes generated by newest vertex bisection.

Below $H^{1}(\G_D)$ the data term is still controlled, only less explicitly.
For $0<t<\frac12$ let
$H^{1/2+t}(\G_D):=[H^{1/2}(\G_D),H^{1}(\G_D)]_{2t}$ be the interpolation
space, let $\phi\in H^{1/2+t}(\G_D)$, and take $w_h=S_h\phi$ in
Lemma~\ref{lem:oscD}.  The operator $I-S_h$ is bounded from $H^{1/2}(\G_D)$ to
$H^{1/2}(\G_D)$ \cite[eq.~(3.10)]{AFKPP:13}, and from $H^{1/2}(\G_D)$ to the
$h_F^{-1}$-weighted $L^{2}(\G_D)$, by interpolation between its local
$L^{2}$-stability and its first-order approximation property
\cite[eqs.~(3.7), (3.9)]{AFKPP:13}.  From $H^{1}(\G_D)$ to either space its
norm is at most $Ch_{\max}^{1/2}$, by \eqref{eq:SZ_bdry}, the
$H^{1}(\G_D)$-stability of $S_h$ \cite[eq.~(3.8)]{AFKPP:13}, and
$\|v\|_{1/2,\G_D}\le C\|v\|_{0,\G_D}^{1/2}\|v\|_{1,\G_D}^{1/2}$.  Interpolating
between the two cases gives
$$
\|\phi-Q_h\phi\|_{1/2,\G_D}
+\Big(\sum_{F\in\cE_D}h_F^{-1}\|\phi-Q_h\phi\|^{2}_{0,F}\Big)^{1/2}
\le Ch_{\max}^{t}\|\phi\|_{1/2+t,\G_D}.
$$
This is a rate in terms of the largest mesh size.
\end{rem}

\section{Non-intrusive a posteriori error estimation}
\label{sec:nonintrusive}

Theorem~\ref{thm:aposteriori} holds for every pair in
$\Sigma\times V^{D}_{\ast}$, whether or not it satisfies the boundary
conditions.  It therefore applies to approximations computed by other methods.  Let $\hat u_h\in V_h$ be such an approximation, for example the
Galerkin approximation of \eqref{weak} in $V_h$ with the Dirichlet datum
replaced by an interpolant or a projection of $\phi$ in $W^{D}_{h}$; in general
$\tr^{g}_{\G_D}(\hat u_h)\neq\phi$.  As in the non-intrusive estimator of
\cite{LiZhang:25}, a flux is recovered from $\hat u_h$ by minimizing the
functional with the scalar fixed: find $\hat\bsigma_h\in\Sigma_h$ such that
\beq\label{eq:recovery}
J(\hat\bsigma_h,\hat u_h;f_1,\bff_2,\phi,\psi)
=\min_{\btau_h\in\Sigma_h}J(\btau_h,\hat u_h;f_1,\bff_2,\phi,\psi),
\eeq
and let $\hat\eta:=J(\hat\bsigma_h,\hat u_h;f_1,\bff_2,\phi,\psi)^{1/2}$.
The bilinear form of the normal equations of \eqref{eq:recovery} is
$(A^{-1}\btau,\bchi)_{0,\O}+(\gradt\btau,\gradt\bchi)_{0,\O}
+(\btau\cdot\bn,\bchi\cdot\bn)_{0,\G_N}$, whose value at $\bchi=\btau$ is at
least $\min\{\Lambda_1^{-1},1\}\|\btau\|^{2}_{H(\divvr;\O)}$, so
\eqref{eq:recovery} has a unique solution; its matrix is the flux block of the
matrix of the normal equations of \eqref{eq:ls_min}.  The Dirichlet term of
$\hat\eta^{2}$, $\sum_{F\in\cE_D}h_F^{-1}\|\phi-\hat u_h\|^{2}_{0,F}$, does not
depend on the recovered flux and is the computable violation of the Dirichlet
condition.  As in Section~\ref{sec:aposteriori}, $\hat\eta^{2}$ is a sum of
element contributions, which serve as local error indicators.

\begin{thm}[non-intrusive error equivalence]\label{thm:nonintrusive}
Let the uniqueness assumption, Assumption~\ref{assum:infsup}, hold, let the
data satisfy Assumption~\ref{assum:data}, let $(\bsigma,u)$ be the solution of
\eqref{fos}, let $Q_h$ satisfy \eqref{def:Qh}, let $\hat u_h\in V_h$, and let
$\hat\bsigma_h$ be the solution of \eqref{eq:recovery}.  Then
\beq\label{eq:ni_pair}
c\,[\hat\eta+\|\phi-Q_h\phi\|_{1/2,\G_D}]
\le\tri(\bsigma-\hat\bsigma_h,u-\hat u_h)\tri
\le C\,[\hat\eta+\|\phi-Q_h\phi\|_{1/2,\G_D}],
\eeq
with the constants of Theorem~\ref{thm:aposteriori}.  The error of $\hat u_h$
alone, $\tri(0,u-\hat u_h)\tri^{2}=\|u-\hat u_h\|^{2}_{1,\O}
+\sum_{F\in\cE_D}h_F^{-1}\|\phi-\hat u_h\|^{2}_{0,F}$, satisfies
\beq\label{eq:ni_scalar}
\tri(0,u-\hat u_h)\tri\le C\,[\hat\eta+\|\phi-Q_h\phi\|_{1/2,\G_D}],
\qquad
\hat\eta\le C_2^{1/2}\Big(\tri(0,u-\hat u_h)\tri^{2}
+\inf_{\btau_h\in\Sigma_h}\tri(\bsigma-\btau_h,0)\tri^{2}\Big)^{1/2}.
\eeq
\end{thm}

\begin{proof}
Since $(\hat\bsigma_h,\hat u_h)\in\Sigma_h\times V_h\subset\Sigma\times V^{D}_{\ast}$,
\eqref{eq:ni_pair} is Theorem~\ref{thm:aposteriori}.  The norm \eqref{def:norm}
splits as $\tri(\btau,v)\tri^{2}=\tri(\btau,0)\tri^{2}+\tri(0,v)\tri^{2}$, so
the first bound in \eqref{eq:ni_scalar} follows from \eqref{eq:ni_pair}.  For
every $\btau_h\in\Sigma_h$, \eqref{eq:recovery}, \eqref{eq:eta_identity} and
Lemma~\ref{lem:continuity} give
$\hat\eta\le J(\btau_h-\bsigma,\hat u_h-u;0)^{1/2}
\le C_2^{1/2}\tri(\bsigma-\btau_h,u-\hat u_h)\tri$, and the splitting of the
norm gives the second bound.
\end{proof}

\section{Sharp weights and over-penalization}
\label{sec:weights}

The analysis above uses the weights $h_F^{-1}$ on $\G_D$ and $1$ on $\G_N$.
This section considers the weights $h_F^{-\alpha}$ on $\G_D$ and $h_F^{-\beta}$
on $\G_N$.  We write $J_{\alpha,\beta}$ and $\tri\cdot\tri_{\alpha,\beta}$ for the
functional \eqref{def:J} and the norm \eqref{def:norm} with these weights, so
that $J=J_{1,0}$ and $\tri\cdot\tri=\tri\cdot\tri_{1,0}$.  Stability bounds the
exponents from below; larger exponents leave the a priori estimate unchanged up
to the approximation of the Dirichlet datum in the discrete trace space.

\begin{prop}[the admissible weights for stability]\label{rem:other_exponents}
Let the uniqueness assumption, Assumption~\ref{assum:infsup}, hold.  If
$\alpha\ge1$ and $\beta\ge0$, then Theorem~\ref{thm:threshold} holds for
$J_{\alpha,\beta}$ and $\tri\cdot\tri_{\alpha,\beta}$ with the same constants.
If $\alpha<1$ or $\beta<0$, no constant $C_1$ independent of the mesh size
satisfies the lower bound in \eqref{eq:normequiv}, already on quasi-uniform
meshes.
\end{prop}

\begin{proof}
Since $h_F\le1$, the weights satisfy $h_F^{-\alpha}\ge h_F^{-1}$ and
$h_F^{-\beta}\ge1$ for $\alpha\ge1$ and $\beta\ge0$.  The proof of
Theorem~\ref{thm:threshold} uses the boundary terms through \eqref{eq:inverseD} and
\eqref{eq:embedN}, which therefore apply with the larger weights, and the boundary
terms of $J_{\alpha,\beta}(\cdot;0)$ and $\tri\cdot\tri_{\alpha,\beta}$ coincide.

For $\alpha<1$, let $g_h\in W^{D}_{h}$ be a nodal basis function of a node
interior to $\G_D$, supported on one patch of faces of diameter
$\simeq h_{\max}$.  Scaling gives $\|g_h\|^{2}_{0,\G_D}\simeq h_{\max}^{d-1}$ and
$\|g_h\|^{2}_{1/2,\G_D}\gtrsim h_{\max}^{d-2}$.  Let $(\btau,v)$ solve
\eqref{fos} with $f_1=0$, $\bff_2=\bzero$, $\phi=g_h$ and $\psi=0$; then
$(\btau,v)\in\Sigma\times V^{D}_{\ast}$, both volume residuals vanish, and
$$
J_{\alpha,\beta}(\btau,v;0)=\sum_{F\in\cE_D}h_F^{-\alpha}\|g_h\|^{2}_{0,F}
\simeq h_{\max}^{d-1-\alpha},
\qquad
\tri(\btau,v)\tri^{2}_{\alpha,\beta}\ge\|v\|^{2}_{1,\O}\ge\|g_h\|^{2}_{1/2,\G_D}
\gtrsim h_{\max}^{d-2}
$$
by \eqref{eq:quotient_lift}.  The ratio of the second to the first is
$\gtrsim h_{\max}^{\alpha-1}$, which is unbounded.

For $\beta<0$, let $(\btau,v)$ solve \eqref{fos} with $f_1=0$, $\bff_2=\bzero$,
$\phi=0$ and $\psi\equiv1$; then $(\btau,v)\in\Sigma\times V^{D}_{\ast}$ and
$J_{\alpha,\beta}(\btau,v;0)=\sum_{F\in\cE_N}h_F^{-\beta}|F|\le h_{\max}^{-\beta}|\G_N|$,
which tends to zero, whereas
$\tri(\btau,v)\tri_{\alpha,\beta}\ge\|\btau\|_{H(\divvr;\O)}\ge\|1\|_{\Hm00(\G_N)}$ by
\eqref{eq:hdiv_trace}, a fixed positive number.
\end{proof}

In particular, the weight $h_F$ on $\G_N$, that is, $\beta=-1$, which the
scaling of $\Hm00(\G_N)$ suggests, destroys the norm equivalence.

Multiplying the Dirichlet and the Neumann boundary terms of the functional by
fixed constants $c_D,c_N>0$, with the norm \eqref{def:norm} unchanged, changes
only the constants: $C_1$ scales with $\min\{1,c_D,c_N\}$ and $C_2$ with
$\max\{1,c_D,c_N\}$.  Every choice $c_D,c_N>0$ is admissible.

Larger exponents are admissible for approximation as well.  On $\G_N$ this is
the mechanism of Section~\ref{sec:apriori}: Lemma~\ref{lem:pythagoras} holds face
by face and each weight is constant on its face, so for every
$\btau_h\in\Sigma_h$
\beq\label{eq:pythagoras_N}
\sum_{F\in\cE_N}h_F^{-\beta}\|\btau_h\cdot\bn-\psi\|^{2}_{0,F}
=\sum_{F\in\cE_N}h_F^{-\beta}\|\btau_h\cdot\bn-\Pi_k\psi\|^{2}_{0,F}
+\sum_{F\in\cE_N}h_F^{-\beta}\|\psi-\Pi_k\psi\|^{2}_{0,F}.
\eeq
On $\G_D$ the same mechanism uses the discrete trace space: let $P^{\alpha}_h\phi\in W^{D}_{h}$ be the orthogonal
projection of $\phi$ with respect to
$\sum_{F\in\cE_D}h_F^{-\alpha}(\cdot,\cdot)_{0,F}$.  For every $v_h\in V_h$,
$\tr^{g}_{\G_D}(v_h)-P^{\alpha}_h\phi\in W^{D}_{h}$, hence
\beq\label{eq:pythagoras_D}
\sum_{F\in\cE_D}h_F^{-\alpha}\|v_h-\phi\|^{2}_{0,F}
=\sum_{F\in\cE_D}h_F^{-\alpha}\|v_h-P^{\alpha}_h\phi\|^{2}_{0,F}
+\sum_{F\in\cE_D}h_F^{-\alpha}\|P^{\alpha}_h\phi-\phi\|^{2}_{0,F},
\eeq
and the minimizer of $J_{\alpha,\beta}$ over $\Sigma_h\times V_h$ is unchanged
when $\phi$ is replaced by $P^{\alpha}_h\phi$.  The projection is global,
because $W^{D}_{h}$ is continuous across faces.

\begin{thm}[over-penalization]\label{rem:overpenalize}
Let the uniqueness assumption, Assumption~\ref{assum:infsup}, hold, let the
data satisfy Assumption~\ref{assum:data}, let the regularity of
Theorem~\ref{thm:apriori} hold for some $0<s\le k+1$, and let $\alpha\ge1$ and
$\beta\ge0$.  Then the minimizer $(\bsigma_h,u_h)$ of
$J_{\alpha,\beta}(\cdot;f_1,\bff_2,\phi,\psi)$ over $\Sigma_h\times V_h$
satisfies
\beq\label{eq:apriori_overpen}
\|u-u_h\|_{1,\O}+\|\bsigma-\bsigma_h\|_{H(\divvr;\O)}
\le C\Big(h_{\max}^{s}\big(|u|_{1+s,\O}+|\bsigma|_{s,\O}+|\gradt\bsigma|_{s,\O}\big)
+\Big(\sum_{F\in\cE_D}h_F^{-1}\|\phi-P^{\alpha}_h\phi\|^{2}_{0,F}\Big)^{1/2}\Big),
\eeq
with $C$ independent of $\alpha$, $\beta$ and the mesh size.
\end{thm}

\begin{proof}
By \eqref{eq:pythagoras_N} and \eqref{eq:pythagoras_D}, $z_h:=(\bsigma_h,u_h)$
minimizes $J_{\alpha,\beta}(\cdot):=J_{\alpha,\beta}(\cdot;f_1,\bff_2,P^{\alpha}_h\phi,\Pi_k\psi)$
over $\Sigma_h\times V_h$.  Let $L_h$ be the nodal lifting from the proof of
Lemma~\ref{lem:inverseD}, and put
$v_h^{\ast}:=I_hu+L_h\big(P^{\alpha}_h\phi-\tr^{g}_{\G_D}(I_hu)\big)$ and
$z_h^{\ast}:=(\Pi^{RT}_h\bsigma,v_h^{\ast})$.  Then
$\tr^{g}_{\G_D}(v_h^{\ast})=P^{\alpha}_h\phi$ and
$\Pi^{RT}_h\bsigma\cdot\bn=\Pi_k\psi$ on $\G_N$, so both boundary residuals of
$z_h^{\ast}$ vanish, whatever the weights.  As in the proof of
Theorem~\ref{thm:apriori}, Proposition~\ref{rem:other_exponents} gives
$\|u_h-v_h^{\ast}\|_{1,\O}+\|\bsigma_h-\Pi^{RT}_h\bsigma\|_{H(\divvr;\O)}
\le2\sqrt2\,C_1^{-1/2}J_{\alpha,\beta}(z_h^{\ast})^{1/2}$, and, since
$(\bsigma,u)$ solves \eqref{fos}, $J_{\alpha,\beta}(z_h^{\ast})^{1/2}$ is
bounded by
$C\big(\|\bsigma-\Pi^{RT}_h\bsigma\|_{H(\divvr;\O)}+\|u-v_h^{\ast}\|_{1,\O}\big)$.
The proof of Lemma~\ref{lem:inverseD} gives
$\|L_hg\|_{1,\O}\le C\big(\sum_{F\in\cE_D}h_F^{-1}\|g\|^{2}_{0,F}\big)^{1/2}$ for
$g\in W^{D}_{h}$, and $\phi=\tr^{g}_{\G_D}(u)$, so
$$
\|u-v_h^{\ast}\|_{1,\O}\le\|u-I_hu\|_{1,\O}
+C\Big(\sum_{F\in\cE_D}h_F^{-1}\|u-I_hu\|^{2}_{0,F}\Big)^{1/2}
+C\Big(\sum_{F\in\cE_D}h_F^{-1}\|\phi-P^{\alpha}_h\phi\|^{2}_{0,F}\Big)^{1/2}.
$$
The triangle inequality, \eqref{eq:RT_approx} and \eqref{eq:SZ_approx} give
\eqref{eq:apriori_overpen}.
\end{proof}

The weights enter neither the constant nor the interpolation errors in
\eqref{eq:apriori_overpen}: the comparison pair meets the projected data
exactly, and the penalty multiplies zero.  The last term of
\eqref{eq:apriori_overpen} is a data term.  It depends only on $\phi$, the
mesh and the projection $P^{\alpha}_h$ onto $W^{D}_{h}$, like the data term in
Theorem~\ref{thm:aposteriori}.  For $\alpha=1$,
$P^{1}_h=P_h$ is the best approximation in this weighted norm, the term is at
most $C h_{\max}^{s}|u|_{1+s,\O}$ by \eqref{eq:SZ_approx}, and
Theorem~\ref{thm:apriori} is recovered.  On $\G_N$ no data term arises,
because $\Pi_k\psi$ is independent of $\beta$.  As in
Remark~\ref{rem:zero_defect}, the estimate holds in the norm of
$H(\divvr;\O)\times H^{1}(\O)$; in $\tri\cdot\tri_{\alpha,\beta}$ the boundary
terms of the error are at least
$\sum_{F\in\cE_D}h_F^{-\alpha}\|\phi-P^{\alpha}_h\phi\|^{2}_{0,F}$ and
$\sum_{F\in\cE_N}h_F^{-\beta}\|\psi-\Pi_k\psi\|^{2}_{0,F}$, by
\eqref{eq:pythagoras_D} and \eqref{eq:pythagoras_N}.

\begin{rem}[minimum, over-penalization, and conditioning]\label{rem:three_roles}
Three questions that are often discussed together are logically distinct.
\emph{(i) Stability asks for a lower threshold.}  For the present boundary
residuals the sharp thresholds are
$h_F^{-1}$ on $\G_D$ and $1$ on $\G_N$.  \emph{(ii) Approximation asks whether
the penalized defect is actually being approximated.}  If a compatible comparison annihilates that defect, its weight is irrelevant:
``large weight times zero is zero.''  Theorem~\ref{rem:overpenalize} proves this
for all $\alpha\ge1$ and $\beta\ge0$; on $\G_D$ only the approximation of the
datum in $W^{D}_{h}$ remains.  The same mechanism is familiar in discontinuous
Galerkin methods.  In the weakly over-penalized symmetric interior-penalty
method of Brenner--Owens--Sung \cite{BOS:08}, the projected scalar jump is
penalized more strongly than in the standard SIP scaling, yet optimal energy
and $L^{2}$ convergence are retained.  \emph{(iii) Conditioning asks for the
algebraic price of that choice.}  A penalty may be harmless for convergence and still make the stiffness matrix
poorly scaled.
\end{rem}

\begin{rem}[large boundary weights beyond finite elements]\label{rem:nn_penalty}
The mechanism is not specific to finite elements.  On $\G_N$ the unweighted
$L^{2}$ penalty controls the natural norm of every square-integrable residual,
by \eqref{eq:embedN}, so the weight $1$ suffices for every trial class.  In
every least-squares minimization, a weight above the stability threshold can
be increased without changing the error bound whenever the trial class
contains an accurate comparison function that meets the boundary data, because
the penalty multiplies zero.  The boundary residuals of this paper can also be
used in the neural-network method FOSLS-deRhaNN \cite{Zhang:26}.
\end{rem}

\section{Extension to a fully discontinuous least-squares method}

\label{sec:dlsfem}

For a fully discontinuous least-squares discretization we prove the error
equivalence of Section~\ref{sec:functional} for every discrete pair, through
conforming companions and Theorem~\ref{thm:aposteriori}, and the norm
equivalence on the discrete spaces only.  The interior jump penalties, like the boundary
penalties of Section~\ref{sec:weights}, may be arbitrarily stronger than the
stability minimum without entering the a priori estimate, because a conforming
comparison pair has vanishing jumps.

The adjective ``fully discontinuous'' is important here.  The broken spaces
below have independent traces on the two sides of every interior face.
Nonconforming spaces such as the Crouzeix--Raviart space, in which a face
moment is already continuous, are a different case.

Under the regularity of Theorem~\ref{thm:apriori} with $0<s\le\frac12$, the
normal trace of the exact flux on $\partial K$ is in general not square
integrable, so a face-wise $L^2$ normal-jump functional is not defined on the full broken space
$\prod_K H(\divvr;K)$.  The method and its analysis use only jumps of discrete
functions.  The discontinuous functional is evaluated on the finite element spaces, where every normal trace
is a polynomial and hence belongs to $L^2(F)$, and at the difference of a
discrete pair and the exact solution.  At such a difference the interior jumps
are those of the discrete pair, because the exact solution has no jumps across
interior faces, and the boundary traces of the exact solution are the data; on
$\G_N$ its normal trace is $\psi\in L^{2}(\G_N)$.

\subsection{Broken finite element spaces and the discrete functional}
Let $\cE_I$ denote the set of interior faces of $\cT_h$.  For every
$F=\partial K^+\cap\partial K^-\in\cE_I$ fix once and for all a unit normal
$\bn_F$ pointing from $K^+$ to $K^-$.  For a broken scalar and flux define on $F$
$[v_h]:=v_h^+-v_h^-$ and $[\btau_h\cdot\bn_F]:=(\btau_h^+-\btau_h^-)\cdot\bn_F$.
Set
\begin{align*}
 \Sigma_h^{\rm b}
 &:=\{\btau_h\in L^2(\O)^d:\ \btau_h|_K\in\bRT_k(K)\ \forall K\in\cT_h\},\\
 V_h^{\rm b}
 &:=\{v_h\in L^2(\O):\ v_h|_K\in P_{k+1}(K)\ \forall K\in\cT_h\}.
\end{align*}
The conforming spaces of Section~\ref{sec:prelim} are subspaces,
$\Sigma_h\subset\Sigma_h^{\rm b}$ and $V_h\subset V_h^{\rm b}$.
We write $\nabla_h$ and $\divvr_h$ for the element-wise gradient and
divergence, and we define the broken norms
\[
 \|v\|_{1,h}^2:=\sum_{K\in\cT_h}\|v\|_{1,K}^2,
 \qquad
 \|\btau\|_{H(\divvr,h)}^2
 :=\sum_{K\in\cT_h}
 \big(\|\btau\|_{0,K}^2+\|\divvr\btau\|_{0,K}^2\big).
\]

For $(\btau_h,v_h)\in\Sigma_h^{\rm b}\times V_h^{\rm b}$ define
\begin{align}
\label{def:Jdisc}
J_{\rm dis}(\btau_h,v_h;f_1,\bff_2,\phi,\psi)
:={}&\|A^{-\half}(\btau_h+A\nabla_hv_h+\bb v_h+\bff_2)\|_{0,\O}^2
+\|\divvr_h\btau_h+\bbeta\cdot\nabla_hv_h+\mu v_h-f_1\|_{0,\O}^2
\nonumber\\
&+\sum_{F\in\cE_I}h_F^{-1}(\|[v_h]\|_{0,F}^2+\|[\btau_h\cdot\bn_F]\|_{0,F}^2)
+\sum_{F\in\cE_D}h_F^{-1}\|v_h-\phi\|_{0,F}^2
+\|\btau_h\cdot\bn-\psi\|_{0,\G_N}^2 .
\end{align}
All six terms are well defined because the arguments are piecewise polynomials
and the data satisfy Assumption~\ref{assum:data}.  As before,
$J_{\rm dis}(\btau_h,v_h;0)$ denotes the functional with homogeneous
data.

The corresponding discrete norm is
\beq\label{def:normdisc}
\tri(\btau_h,v_h)\tri_{{\rm d},h}^2
:=\|v_h\|_{1,h}^2+\|\btau_h\|_{H(\divvr,h)}^2
+\sum_{F\in\cE_I}h_F^{-1}(\|[v_h]\|_{0,F}^2+\|[\btau_h\cdot\bn_F]\|_{0,F}^2)
+\sum_{F\in\cE_D}h_F^{-1}\|v_h\|_{0,F}^2+\|\btau_h\cdot\bn\|_{0,\G_N}^2 .
\eeq

\subsection{Conforming companion operators}
We use the standard averaging operators \cite[Sect.~22.2]{EG:21}, which
average the degrees of freedom shared by several elements; a normal moment on a
boundary face belongs to one element and is left unchanged.

\begin{lem}[discrete conforming companions]\label{lem:disc_companion}
There are linear maps $\cC_h^u:V_h^{\rm b}\longrightarrow V_h$ and $\cC_h^\sigma:\Sigma_h^{\rm b}\longrightarrow\Sigma_h$ with constants depending only on shape regularity, $k$, and $d$, such that
\begin{align}
\label{eq:scalar_companion}
 &\|v_h-\cC_h^u v_h\|_{1,h}^2
 +\sum_{F\in\cE_D}h_F^{-1}
   \|v_h-\cC_h^u v_h\|_{0,F}^2
 \ \lesssim\
 \sum_{F\in\cE_I}h_F^{-1}\|[v_h]\|_{0,F}^2,\\
\label{eq:flux_companion}
 &\|\btau_h-\cC_h^\sigma\btau_h\|_{H(\divvr,h)}^2
 \ \lesssim\
 \sum_{F\in\cE_I}h_F^{-1}
 \|[\btau_h\cdot\bn_F]\|_{0,F}^2,
\end{align}
and the flux companion preserves the normal trace on boundary faces,
\beq\label{eq:flux_companion_boundary}
 (\cC_h^\sigma\btau_h)\cdot\bn
 =\btau_h\cdot\bn
 \qquad\text{on every }F\subset\partial\O.
\eeq
\end{lem}

\begin{proof}
The volume estimates are \cite[Lem.~22.3]{EG:21}.  The boundary term in
\eqref{eq:scalar_companion} follows from the $L^{2}(K_F)$ estimate there and
the discrete trace inequality on $K_F$, and
\eqref{eq:flux_companion_boundary} holds because the boundary normal moments
are not averaged.
\end{proof}

\subsection{Norm equivalence}

\begin{thm}[discrete norm equivalence]\label{thm:disc_normequiv}
Let the uniqueness assumption, Assumption~\ref{assum:infsup}, hold.  Then
\beq\label{eq:disc_normequiv}
 C_{\rm d}\tri(\btau_h,v_h)\tri_{{\rm d},h}^2
 \ \le\ J_{\rm dis}(\btau_h,v_h;0)
 \ \le\ C_{\rm d}'\tri(\btau_h,v_h)\tri_{{\rm d},h}^2
\eeq
for every $(\btau_h,v_h)\in\Sigma_h^{\rm b}\times V_h^{\rm b}$.
The constants are independent of the mesh size.  Consequently
\eqref{def:Jdisc} has a unique minimizer for all data satisfying
Assumption~\ref{assum:data}.
\end{thm}

\begin{proof}
The upper bound follows from the triangle inequality, since the jump and
boundary terms of \eqref{def:Jdisc} and \eqref{def:normdisc} coincide.  For the
lower bound let $v_h^c:=\cC_h^u v_h\in V_h$ and
$\btau_h^c:=\cC_h^\sigma\btau_h\in\Sigma_h$.  The volume residuals of
$(\btau_h^c,v_h^c)$ and $(\btau_h,v_h)$ differ by the companion errors, the
Neumann terms coincide by \eqref{eq:flux_companion_boundary}, and the Dirichlet
terms differ by the boundary term of \eqref{eq:scalar_companion}, so
Lemma~\ref{lem:disc_companion} gives
$J(\btau_h^c,v_h^c;0)\lesssim J_{\rm dis}(\btau_h,v_h;0)$.
Theorem~\ref{thm:threshold} bounds $\|v_h^c\|_{1,\O}$ and
$\|\btau_h^c\|_{H(\divvr;\O)}$, Lemma~\ref{lem:disc_companion} the broken
volume norms of $(\btau_h,v_h)$, and the jump and boundary terms of
\eqref{def:normdisc} are terms of \eqref{def:Jdisc}.  Positive definiteness
gives the unique minimizer.
\end{proof}

The weight $h_F^{-1}$ on both interior jumps is also necessary for the norm
equivalence on the unconstrained broken spaces.

\begin{prop}[sharpness of the jump weights]\label{prop:disc_sharp}
Let $\O=(0,1)^{d}$, let $\G_D$ be the union of the interiors of the faces
$x_1=0$ and $x_1=1$, let $A=I$, $\bb=\bbeta=\bzero$ and $\mu=0$, and let
$k\ge0$.  Let the meshes be quasi-uniform, resolve the planes $x_1=i/M$,
$i=1,\dots,M-1$, and satisfy $h_{\max}\simeq M^{-1}$.
Replace the interior weights in \eqref{def:Jdisc} by constants $\omega_u,\omega_\sigma>0$
on the scalar and the flux jumps.  If
\beq\label{eq:disc_sharp}
\|v_h\|_{1,h}^{2}+\|\btau_h\|_{H(\divvr,h)}^{2}\le C\,J_{\rm dis}(\btau_h,v_h;0)
\qquad\forall(\btau_h,v_h)\in\Sigma_h^{\rm b}\times V_h^{\rm b}
\eeq
with $C$ independent of the mesh, then $\omega_\sigma\ge c\,h_{\max}^{-1}$ and
$\omega_u\ge c\,h_{\max}^{-1}$ with $c>0$ independent of the mesh.
\end{prop}

\begin{proof}
\emph{The flux jump.}  Let $w$ be the continuous piecewise linear function on
$[0,1]$ with $w(0)=w(1)=0$ and slope $m_j:=\frac{M-1}{2}-j$ on
$(j/M,(j+1)/M)$, $j=0,\dots,M-1$; the slopes have mean zero, so $w(1)=0$
indeed.  Put $v_h(x):=w(x_1)$ and $\btau_h:=-\nabla v_h$.  The meshes resolve
the planes, so $v_h$ is affine on every element and $v_h\in V_h$, while
$\btau_h$ is constant on every element and $\btau_h\in\Sigma_h^{\rm b}$.  Both
volume residuals of $(\btau_h,v_h)$ vanish, $v_h$ has no jumps and vanishes on
$\G_D$, and $\btau_h\cdot\bn=0$ on $\G_N$ because $\btau_h$ is parallel to the
$x_1$-axis.  The normal jump of $\btau_h$ vanishes inside the slabs and has
modulus $|m_{i-1}-m_i|=1$ on the planes, each of area one.  Hence
$J_{\rm dis}(\btau_h,v_h;0)=\omega_\sigma(M-1)$, whereas
$\|\nabla v_h\|_{0,\O}^{2}=\frac1M\sum_{j=0}^{M-1}m_j^{2}=\frac{M^{2}-1}{12}$.
Inserting $(\btau_h,v_h)$ into \eqref{eq:disc_sharp} gives
$\omega_\sigma\ge(M+1)/(12C)$.

\emph{The scalar jump.}  Fix $g\in C_c^\infty(\O)$ with $g\ne0$, put
$\btau_h:=\bzero$, and let $v_h$ be the element-wise constant function with
value $g$ at the barycenter of each element.  Since the support of $g$ has
positive distance to $\p\O$, for $h_{\max}$ small $v_h$
vanishes on the elements touching $\p\O$, so the boundary residuals vanish, and
both volume residuals vanish because $\nabla_hv_h=\bzero$.  On every interior
face $|[v_h]|\le 2h_{\max}\|\nabla g\|_{L^{\infty}(\O)}$, since the barycenters of
two neighboring elements are at distance at most $2h_{\max}$, and the total area
of the interior faces of a quasi-uniform mesh is at most $Ch_{\max}^{-1}$, so
$J_{\rm dis}(\bzero,v_h;0)=\omega_u\sum_{F\in\cE_I}\|[v_h]\|_{0,F}^{2}
\le C\omega_u h_{\max}\|\nabla g\|^{2}_{L^{\infty}(\O)}$, whereas
$\|v_h\|_{0,\O}\to\|g\|_{0,\O}>0$.  Inserting $(\bzero,v_h)$ into
\eqref{eq:disc_sharp} gives $\omega_u\ge c\,h_{\max}^{-1}$.
\end{proof}

Both examples rest on the same fact: the interior faces fill the domain, and
on quasi-uniform meshes their total area is of order $h_{\max}^{-1}$.  A flux
jump acts on a test function $v$ through
$$
\sum_{F\in\cE_I}([\btau_h\cdot\bn_F],v)_{0,F}
\le\Big(\sum_{F\in\cE_I}h_F^{-1}\|[\btau_h\cdot\bn_F]\|^{2}_{0,F}\Big)^{1/2}
\Big(\sum_{F\in\cE_I}h_F\|v\|^{2}_{0,F}\Big)^{1/2},
$$
and the element-wise trace inequality gives
$\sum_{F\in\cE_I}h_F\|v\|^{2}_{0,F}\le C\big(\|v\|^{2}_{0,\O}+h_{\max}^{2}|v|^{2}_{1,h}\big)$;
the weight $h_F^{-1}$ compensates exactly the factor $h_F$ that the trace
inequality on all faces requires.  On the fixed surface $\G_N$ the trace inequality has a
constant independent of the mesh, and the weight $1$ suffices.

\subsection{Error equivalence}
The error equivalence is derived from Theorem~\ref{thm:aposteriori} through
the conforming companions: the jumps compare a discrete pair with its
companions, and Theorem~\ref{thm:aposteriori} compares the companions with the
exact solution.  The argument is that of \cite[Thm.~7.3]{LiangZhang:23} for
the least-squares method with lowest-order Raviart--Thomas fluxes and
Crouzeix--Raviart scalars; \cite[Thm.~12]{Bri:26} follows it for the
over-penalized discontinuous LSFEM for the Poisson problem with homogeneous
boundary conditions.  The error is measured by broken volume norms together
with the jumps of the discrete pair and the two boundary residuals.

\begin{thm}[error equivalence in the broken norm]\label{thm:disc_apost}
Let the uniqueness assumption, Assumption~\ref{assum:infsup}, hold, let the
data satisfy Assumption~\ref{assum:data}, let $(\bsigma,u)$ be the solution of
\eqref{fos} and $Q_h$ a projection satisfying \eqref{def:Qh}, let
$(\btau_h,v_h)\in\Sigma_h^{\rm b}\times V_h^{\rm b}$, and put
$\eta_{\rm d}^2:=J_{\rm dis}(\btau_h,v_h;f_1,\bff_2,\phi,\psi)$.  Then
\begin{align}
\label{eq:disc_apost}
c\,[\eta_{\rm d}+\|\phi-Q_h\phi\|_{1/2,\G_D}]
\le{}&\Big(\|u-v_h\|_{1,h}^2+\|\bsigma-\btau_h\|_{H(\divvr,h)}^2
+\sum_{F\in\cE_I}h_F^{-1}(\|[v_h]\|_{0,F}^2+\|[\btau_h\cdot\bn_F]\|_{0,F}^2)\\
\nonumber
&\ +\sum_{F\in\cE_D}h_F^{-1}\|\phi-v_h\|_{0,F}^2+\|\psi-\btau_h\cdot\bn\|_{0,\G_N}^2\Big)^{1/2}
\le C\,[\eta_{\rm d}+\|\phi-Q_h\phi\|_{1/2,\G_D}],
\end{align}
with constants $c,C>0$ depending on $C_Q$ but independent of the mesh size.
\end{thm}

\begin{proof}
Let $v_h^c:=\cC_h^uv_h$ and $\btau_h^c:=\cC_h^\sigma\btau_h$.  Since
$(\bsigma,u)$ solves \eqref{fos}, the volume residuals in $\eta_{\rm d}^2$ are
bounded by the broken volume norms of the error, and the other terms of
$\eta_{\rm d}^2$ are terms of the middle term of \eqref{eq:disc_apost}; this
bounds $\eta_{\rm d}$.  For the data term take $w_h:=\tr^{g}_{\G_D}(v_h^c)$ in
Lemma~\ref{lem:oscD}: by \eqref{eq:quotient_lift},
$\|\phi-w_h\|_{1/2,\G_D}\le\|u-v_h\|_{1,h}+\|v_h-v_h^c\|_{1,h}$, the weighted
$L^{2}$ norm of $\phi-w_h$ is at most the Dirichlet term of the middle term
plus the boundary term of \eqref{eq:scalar_companion}, and
Lemma~\ref{lem:disc_companion} bounds the companion terms by the scalar jump
term.  This gives the lower bound.  For the upper bound, the jump and boundary
terms of the middle term are terms of $\eta_{\rm d}^2$, and
$\|u-v_h\|_{1,h}+\|\bsigma-\btau_h\|_{H(\divvr,h)}$ is at most
$\|u-v_h^c\|_{1,\O}+\|\bsigma-\btau_h^c\|_{H(\divvr;\O)}$ plus companion terms,
which are at most $C\eta_{\rm d}$ by Lemma~\ref{lem:disc_companion}.
Theorem~\ref{thm:aposteriori}, applied to
$(\btau_h^c,v_h^c)\in\Sigma_h\times V_h$, bounds the first two by
$C\big(J(\btau_h^c,v_h^c;f_1,\bff_2,\phi,\psi)^{1/2}+\|\phi-Q_h\phi\|_{1/2,\G_D}\big)$,
and $J(\btau_h^c,v_h^c;f_1,\bff_2,\phi,\psi)^{1/2}\le C\eta_{\rm d}$ as in the
proof of Theorem~\ref{thm:disc_normequiv}.
\end{proof}

\subsection{A priori error estimate}
Let $(\bsigma_h^{\rm d},u_h^{\rm d})\in\Sigma_h^{\rm b}\times V_h^{\rm b}$
be the minimizer of \eqref{def:Jdisc} with the exact data.  Since every boundary
normal trace of $\Sigma_h^{\rm b}$ belongs to $P_k(F)$, Lemma~\ref{lem:pythagoras}
applies verbatim: replacing $\psi$ by $\Pi_k\psi$ changes the value of the
functional only by a constant, so the minimizer is the same.

\begin{thm}[a priori estimate in the broken norm]
\label{thm:disc_apriori}
Let the uniqueness assumption, Assumption~\ref{assum:infsup}, hold, let the
data satisfy Assumption~\ref{assum:data}, let $(\bsigma,u)$ be the solution of
\eqref{fos}, and assume the regularity of Theorem~\ref{thm:apriori} for some
$0<s\le k+1$.  Then
\beq\label{eq:disc_apriori}
 \|u-u_h^{\rm d}\|_{1,h}
 +\|\bsigma-\bsigma_h^{\rm d}\|_{H(\divvr,h)}
  \le C\,h_{\max}^{s}
 \big(|u|_{1+s,\O}+|\bsigma|_{s,\O}+|\gradt\bsigma|_{s,\O}\big),
\eeq
with $C$ independent of the mesh size.
\end{thm}

\begin{proof}
Let $\psi_h:=\Pi_k\psi$, $z_h:=(\bsigma_h^{\rm d},u_h^{\rm d})$ and
$z_h^\ast:=(\Pi_h^{RT}\bsigma,I_hu)\in\Sigma_h\times V_h$.  By the observation
above, $z_h$ minimizes $J_{\rm dis}(\cdot;f_1,\bff_2,\phi,\psi_h)$, so
$J_{\rm dis}(z_h-z_h^\ast;0)^{1/2}\le2J_{\rm dis}(z_h^\ast;f_1,\bff_2,\phi,\psi_h)^{1/2}$,
and Theorem~\ref{thm:disc_normequiv} bounds the broken volume norms of
$z_h-z_h^\ast$ by this quantity.  The comparison pair is conforming,
\beq\label{eq:disc_zero_jumps}
 [I_hu]=0,
 \qquad
 [\Pi_h^{RT}\bsigma\cdot\bn_F]=0
 \qquad\text{on every }F\in\cE_I,
\eeq
so $J_{\rm dis}(z_h^\ast;f_1,\bff_2,\phi,\psi_h)=J(z_h^\ast;f_1,\bff_2,\phi,\psi_h)$,
which \eqref{eq:apriori_step2} bounds.  The triangle inequality with
\eqref{eq:RT_approx} and \eqref{eq:SZ_approx} gives \eqref{eq:disc_apriori}.
\end{proof}

\begin{rem}[over-penalization]\label{rem:disc_overpen}
The interior weights $h_F^{-1}$ in \eqref{def:Jdisc} may be replaced by any
weights $\omega_{u,F},\omega_{\sigma,F}$ with
$\omega_{u,F},\omega_{\sigma,F}\ge c\,h_F^{-1}$ for a fixed $c>0$, for instance
$h_F^{-\rho}$ with $\rho\ge1$, provided the same weights are used in the norm
\eqref{def:normdisc}.  Theorem~\ref{thm:disc_normequiv} then holds with a lower
constant depending on $c$ only, because its proof uses the jump terms only
through the companion bounds of Lemma~\ref{lem:disc_companion}, which involve
$h_F^{-1}$.  Theorem~\ref{thm:disc_apriori} holds with the same right-hand side
and a constant depending on $c$ only, the same constant if $c\ge1$, because the comparison pair has zero interior jumps
by \eqref{eq:disc_zero_jumps}, so the weights multiply zero.  Theorem~\ref{thm:disc_apost} holds with the
same weights in $\eta_{\rm d}$ and in the norm, by the same argument.
Over-penalization of the interior jumps therefore leaves the convergence rate
unchanged.
\end{rem}

\begin{rem}[weaker interior weights]\label{rem:disc_weaker}
Let the interior weights $h_F^{-1}$ in \eqref{def:Jdisc} and
\eqref{def:normdisc} be replaced by weights $\omega_{u,F},\omega_{\sigma,F}>0$,
and put
$\kappa_h:=\min_{F\in\cE_I}\min\{\omega_{u,F},\omega_{\sigma,F}\}\,h_F$.  The
proof of Theorem~\ref{thm:disc_apriori} uses the norm equivalence only after
the bound
$$
J_{\rm dis}(z_h-z_h^\ast;0)^{1/2}\le2J_{\rm dis}(z_h^\ast;f_1,\bff_2,\phi,\Pi_k\psi)^{1/2}
\le Ch_{\max}^{s}\big(|u|_{1+s,\O}+|\bsigma|_{s,\O}+|\gradt\bsigma|_{s,\O}\big),
$$
whose right-hand side does not depend on the interior weights, because
$z_h^\ast$ has no jumps.  Hence $z_h-z_h^\ast$ is of order $h_{\max}^{s}$ in the
norm induced by the functional, for every choice of interior weights.  The norm
equivalence is needed to pass from this norm to the norm of
\eqref{eq:disc_apriori}.

If $\kappa_h\le1$, every interior term of the functional is at least $\kappa_h$
times the corresponding term of \eqref{def:Jdisc}, so the lower bound in
\eqref{eq:disc_normequiv} holds with the constant $\kappa_hC_{\rm d}$ for the
broken volume norms.  The proof of Theorem~\ref{thm:disc_apriori} then gives
\eqref{eq:disc_apriori} with its right-hand side multiplied by
$1+\kappa_h^{-1/2}$.  For $\omega_{u,F}=\omega_{\sigma,F}=h_F^{-\rho}$ with
$\rho<1$ on quasi-uniform meshes, $\kappa_h\simeq h_{\max}^{1-\rho}$ and the
resulting bound is of order $h_{\max}^{s-(1-\rho)/2}$; for constant weights it
is of order $h_{\max}^{s-1/2}$.  In the setting of
Proposition~\ref{prop:disc_sharp} the lower constant of the norm equivalence is
indeed of order $\kappa_h$ for constant weights, so the factor cannot be removed
from the norm equivalence.

With a constant weight $\omega_\sigma$, the upper bound of
\eqref{eq:disc_apost} fails for the pair $(\btau_h,v_h)$ in the proof of
Proposition~\ref{prop:disc_sharp}.  For the data $f_1=0$,
$\bff_2=\bzero$, $\phi=0$ and $\psi=0$ the exact solution is zero, the error of
$(\btau_h,v_h)$ is $(\btau_h,v_h)$ itself, and
$\eta_{\rm d}^{2}=\omega_\sigma(M-1)$ is at most $6\omega_\sigma/(M+1)$ times
the square of the middle term of \eqref{eq:disc_apost}.  We therefore use the
weight $h_F^{-1}$: it is the weakest weight for which the norm equivalence
holds, and with it the a priori estimate \eqref{eq:disc_apriori} holds without
a mesh-dependent factor and the error equivalence \eqref{eq:disc_apost}
holds for every discrete pair.
\end{rem}

\subsection{Relation with over-penalized discontinuous Galerkin and least-squares methods}
Theorem~\ref{thm:disc_apriori} puts several familiar observations into one framework.
For scalar jumps, the weakly over-penalized symmetric interior-penalty methods
of Brenner, Owens and Sung \cite{BOS:08,BOS:12}, recalled in
Remark~\ref{rem:three_roles}, are a precedent.

Classical discontinuous LSFEMs likewise employ strong face penalties, including
$h_F^{-1}$ normal-flux jump terms; see Ye--Zhang \cite{YZ:19}.  For the
diffusion equation, the discontinuous variant of the method of Bensow--Larson
\cite[Sect.~3.3.1]{BL:05} uses all weights of \eqref{def:Jdisc}, and
Bensow--Larson \cite[Lemma~3.3]{BL:05} already compare
with a conforming interpolant, a Brezzi--Douglas--Marini field together with a
Scott--Zhang interpolant, so that all interior face terms vanish.
Theorem~\ref{thm:disc_apriori} uses the same mechanism for general
second-order elliptic equations with weakly imposed mixed boundary conditions,
and shows why such weights, and any stronger ones as in
Remark~\ref{rem:disc_overpen}, are compatible with optimal approximation.

Bringmann \cite{Bri:26} imposes the extra constraint that the lowest moments
of the normal jumps vanish, which permits the weight $h_F$ on the flux jump but
complicates the computation.
Theorems~\ref{thm:disc_apost} and~\ref{thm:disc_apriori} show that this
constraint is not needed: the method is a pure penalty method with no extra
mechanism, and on the unconstrained spaces with the weight $h_F^{-1}$, the error equivalence holds for every discrete pair and the a
priori estimate is optimal.

For the boundary residuals and the interior jumps alike, stability determines
how much penalty is necessary,
compatibility determines whether a stronger penalty enters the approximation,
and conditioning determines its algebraic cost.  The first of these is sharp
for the weak boundary residuals by Theorem~\ref{thm:threshold} and
Proposition~\ref{rem:other_exponents}.  The second is settled for both boundary
residuals by Theorem~\ref{rem:overpenalize} and, in the fully discontinuous
setting, for both interior jumps by Theorem~\ref{thm:disc_apriori} and
Remark~\ref{rem:disc_overpen}.  Once a
stable method admits a zero-defect comparison, the penalty has already done its
stability job and disappears from the approximation argument.

\begin{rem}[the original Crouzeix--Raviart LSFEM is a different endpoint]
\label{rem:cr_clean}
The Crouzeix--Raviart space carries the built-in continuity
$\int_F[v_h]\,ds=0$, $F\in\cE_I$, which the standard Crouzeix--Raviart method
uses in place of a jump penalty.  The least-squares method of
\cite{LiangZhang:23}, with $H(\divvr)$-conforming fluxes and Crouzeix--Raviart
scalars, was designed in the same spirit: its functional contains only the
element-wise residuals.  For general, possibly indefinite, equations its
analysis uses discrete stability for sufficiently fine meshes
\cite{LiangZhang:23}.  Theorem~\ref{thm:disc_normequiv} concerns completely
unconstrained broken spaces, while the Crouzeix--Raviart space carries a
face-moment constraint, so its zero-penalty stability is a different endpoint.
\end{rem}

\section{Numerical experiments}
\label{sec:numerics}

All experiments use $k=0$, that is, lowest-order Raviart--Thomas fluxes and
piecewise linear scalars, conforming or broken according to the method.  The errors are computed with a quadrature
rule exact for polynomials of degree ten, on elements subdivided geometrically
towards a singular point of the exact solution.  On uniform meshes, rates are
$\log_2$ of the ratio of the values on consecutive levels.  Adaptive meshes are
generated by Dörfler marking with parameter $1/2$ and newest vertex bisection;
there, rates are mean slopes over the last four iterations with respect to the
number $N$ of unknowns, and the optimal rate is $N^{-1/2}$.

\subsection{Weakly imposed boundary conditions}\label{sec:num_weak}
Let $\O=(0,1)^{2}$ with $\G_D:=\{0,1\}\times(0,1)$ and
$\G_N:=(0,1)\times\{0,1\}$, let $A=I$, $\bb=\bzero$,
$\bbeta=(1,\tfrac12)^{T}$, $\mu=1$ and $\bff_2=\bzero$, and choose $f_1$,
$\phi$ and $\psi$ so that $u(x,y)=e^{x}\cos(\pi y)+xy$; the Neumann datum
$\psi(x,0)=x$, $\psi(x,1)=-x$ is not piecewise constant, so
$\psi\neq\Pi_0\psi$.  The initial mesh consists of the four triangles obtained
by joining the center of the square to its vertices, and each refinement
divides every triangle into four, so that on level $\ell$ every boundary face
has $h_F=h_{\max}=2^{-\ell}$.  For the minimizer $(\bsigma_h,u_h)$ of
$J_{\alpha,\beta}$ we report
$$
e_h:=\big(\|u-u_h\|^{2}_{1,\O}+\|\bsigma-\bsigma_h\|^{2}_{H(\divvr;\O)}\big)^{1/2},
\qquad
\eta_{\alpha,\beta}:=J_{\alpha,\beta}(\bsigma_h,u_h;f_1,\bff_2,\phi,\psi)^{1/2},
$$
so that $\eta_{1,0}=\eta_h$, and the square roots $\eta_D$ and $\eta_N$ of the
Dirichlet and Neumann terms of $\eta_{\alpha,\beta}^{2}$.

\begin{figure}[ht]
\centering
\includegraphics[width=0.40\textwidth]{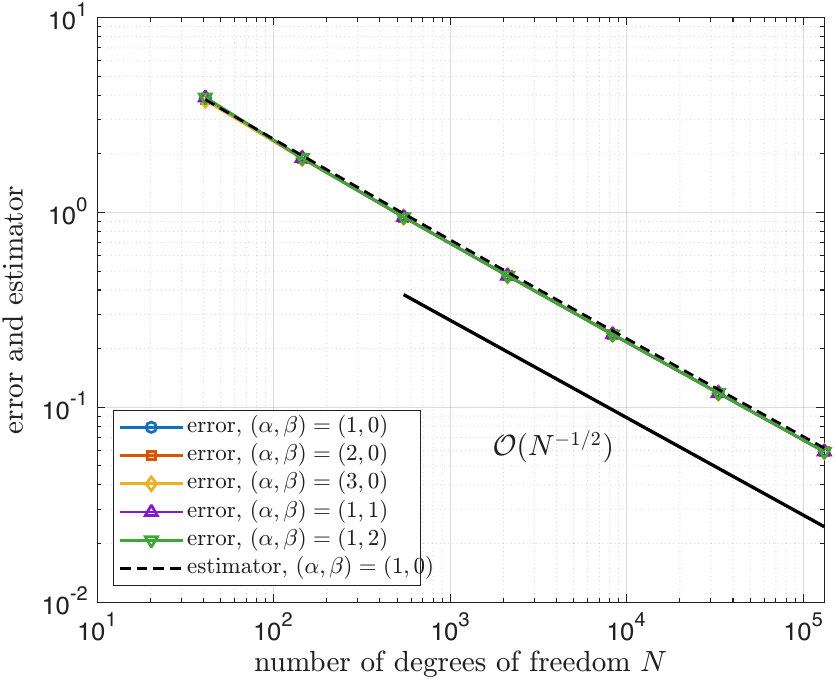}\qquad
\includegraphics[width=0.40\textwidth]{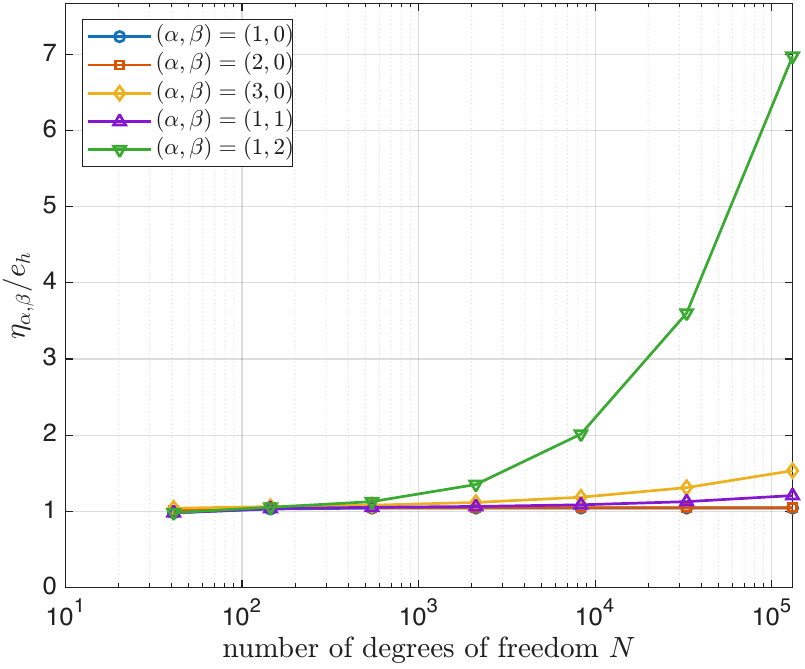}
\caption{Weakly imposed boundary conditions on uniform meshes: the error $e_h$
of the minimizers of $J_{\alpha,\beta}$ and the estimator $\eta_h$ of
\eqref{eq:ls_min} (left), and the ratio $\eta_{\alpha,\beta}/e_h$ (right).}
\label{fig:weakBC}
\end{figure}

For $(\alpha,\beta)=(1,0)$, the method \eqref{eq:ls_min}, the error converges
with the rate $h_{\max}$ of Theorem~\ref{thm:apriori} and $\eta_h/e_h$ settles
at $1.043$, see Figure~\ref{fig:weakBC}; in the norm
$\tri(\bsigma-\bsigma_h,u-u_h)\tri=(e_h^{2}+\eta_D^{2}+\eta_N^{2})^{1/2}$ of
Theorem~\ref{thm:aposteriori} the efficiency index is $1.04$.  The data term
of \eqref{eq:reliability} is of order $h_{\max}^{3/2}$, by
Lemma~\ref{lem:oscD} with $w_h$ the nodal interpolant of $\phi$.

Since all faces in $\cE_D$ have the same size, $P^{\alpha}_h=P_h$, the data term
of \eqref{eq:apriori_overpen} is of order $h_{\max}^{3/2}$, and
Theorem~\ref{rem:overpenalize} gives the rate $h_{\max}$ for every
$\alpha\ge1$ and $\beta\ge0$: the error curves of $(1,0)$, $(2,0)$, $(3,0)$,
$(1,1)$ and $(1,2)$ agree in four digits from level $5$ on, and the difference
between each minimizer and that of $(1,0)$ decreases with rate about two.  The
over-penalized functional, however, is dominated by the boundary data measured
with the larger weights: by \eqref{eq:pythagoras_N} and
\eqref{eq:pythagoras_D}, $\eta_N^{2}\ge h_{\max}^{2-\beta}/6$ and
$\eta_D^{2}\ge\sum_{F\in\cE_D}h_F^{-\alpha}\|\phi-P^{\alpha}_h\phi\|^{2}_{0,F}$,
which is of order $h_{\max}^{4-\alpha}$, and on level $7$ the ratio
$\eta_{\alpha,\beta}/e_h$ is $1.53$ for $(3,0)$, $1.21$ for $(1,1)$ and $6.98$
for $(1,2)$.  This is the distinction of Remark~\ref{rem:zero_defect}: the a
priori estimate is in $H(\divvr;\O)\times H^{1}(\O)$, and the functional with
$(\alpha,\beta)=(1,0)$ is the a posteriori estimator.

The thresholds of Proposition~\ref{rem:other_exponents} are visible in the
smallest eigenvalue
$$
\lambda_{\min}:=\min_{0\neq(\btau_h,v_h)\in\Sigma_h\times V_h}
\frac{J_{\alpha,\beta}(\btau_h,v_h;0)}{\|\btau_h\|^{2}_{H(\divvr;\O)}+\|v_h\|^{2}_{1,\O}},
$$
see Figure~\ref{fig:threshold}: from level $4$ on it lies between $0.30$ and
$0.32$ for $\alpha=1,2$ with $\beta=0$ and for $\beta=0,1$ with $\alpha=1$,
whereas it decreases for $\alpha=0,\frac12$ and for $\beta=-1,-\frac12$, for
$\beta<0$ with a rate approaching $-\beta$, that of the flux example in the
proof of Proposition~\ref{rem:other_exponents}.

\begin{figure}[ht]
\centering
\includegraphics[width=0.40\textwidth]{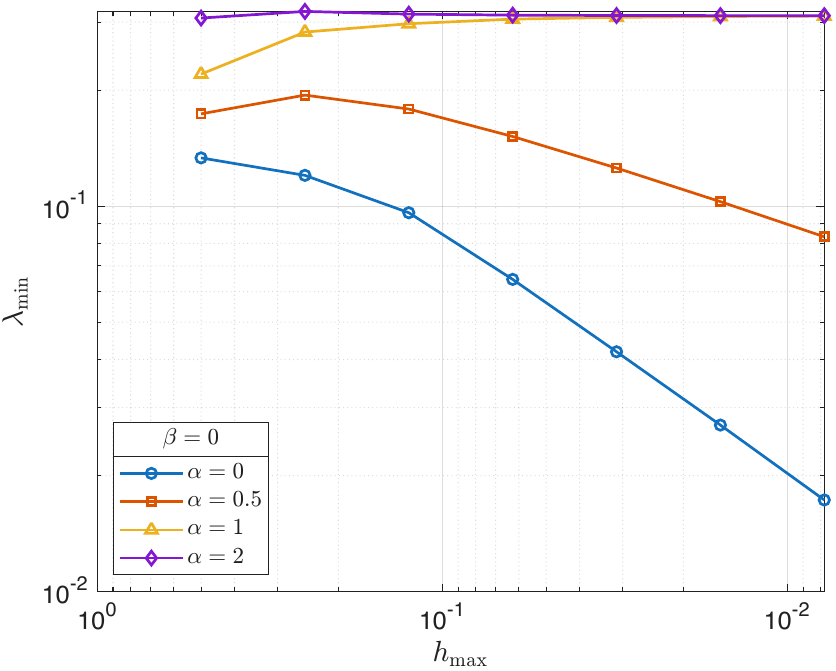}\qquad
\includegraphics[width=0.40\textwidth]{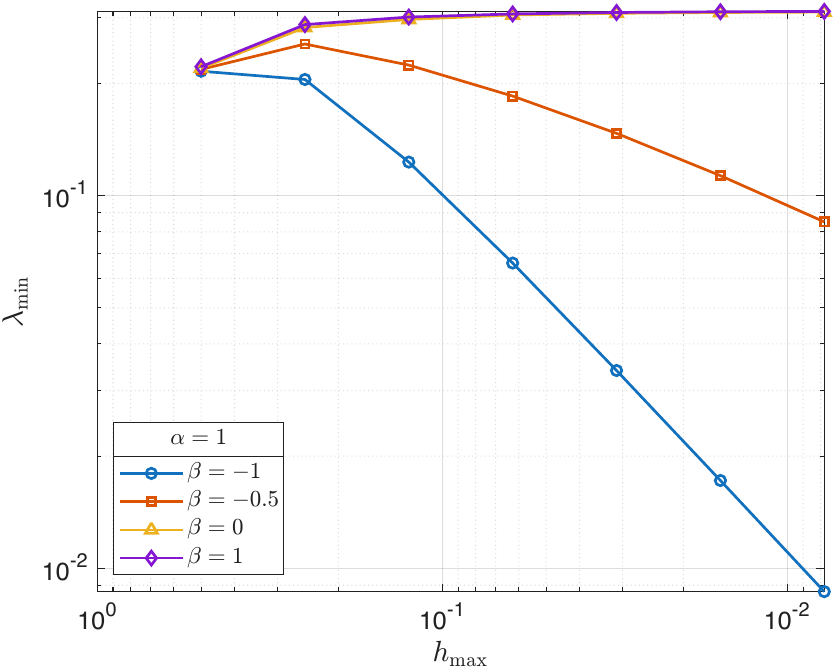}
\caption{The smallest eigenvalue $\lambda_{\min}$ of $J_{\alpha,\beta}(\cdot;0)$
with respect to $\|\cdot\|^{2}_{H(\divvr;\O)}+\|\cdot\|^{2}_{1,\O}$ on
$\Sigma_h\times V_h$, uniform meshes: $\beta=0$ (left) and $\alpha=1$ (right).}
\label{fig:threshold}
\end{figure}

\FloatBarrier
\subsection{Dirichlet data below \texorpdfstring{$H^{1}(\G_D)$}{H1(Gamma\_D)}}\label{sec:num_rough}
We keep the method \eqref{eq:ls_min}, the domain and the initial mesh of
Section~\ref{sec:num_weak}, and let $A=I$, $\bb=\bbeta=\bzero$, $\mu=0$,
$f_1=0$, $\bff_2=\bzero$ and $u=\mathrm{Re}\,z^{\nu}$ with
$z=(y-\tfrac12)+\mathrm{i}\,x$ and $\nu=\tfrac14$, which is singular at
$P=(0,\tfrac12)$.  On the side $x=0$ of $\G_D$ the datum is
$\phi=|y-\tfrac12|^{\nu}$ for $y>\tfrac12$ and
$\phi=\cos(\nu\pi)|y-\tfrac12|^{\nu}$ for $y<\tfrac12$, so
$\phi\in H^{1/2+t}(\G_D)$ for every $t<\nu$ but $\phi\notin H^{1}(\G_D)$, and
Remark~\ref{rem:oscD} applies; moreover $u\in H^{1+t}(\O)$ for $t<\nu$.  We
also report the computable part of the left-hand side of
\eqref{eq:oscD_bound} for $Q_h=P_h$,
$$
d_h:=\Big(\sum_{F\in\cE_D}h_F^{-1}\|\phi-P_h\phi\|^{2}_{0,F}\Big)^{1/2}.
$$

\begin{figure}[ht]
\centering
\includegraphics[width=0.40\textwidth]{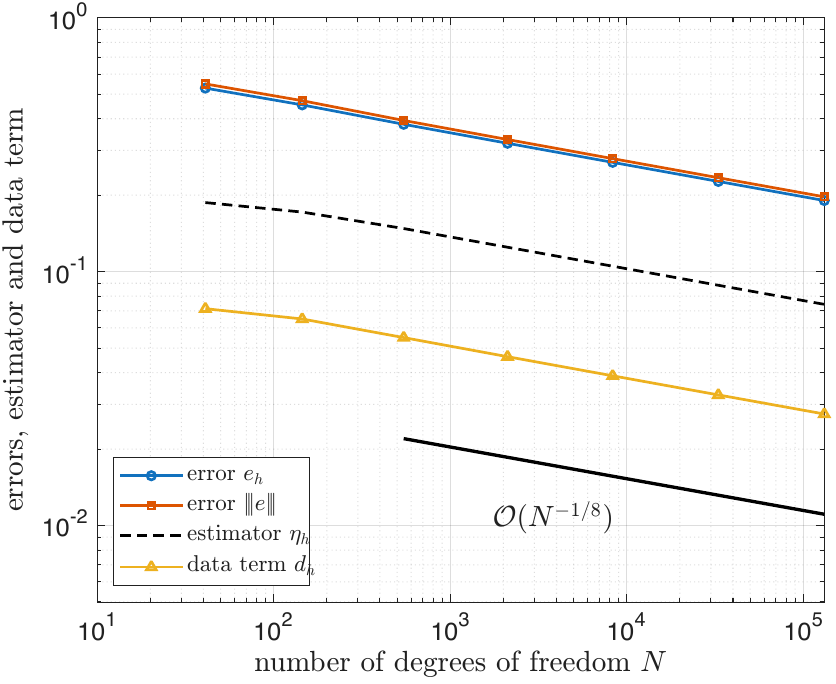}\qquad
\includegraphics[width=0.40\textwidth]{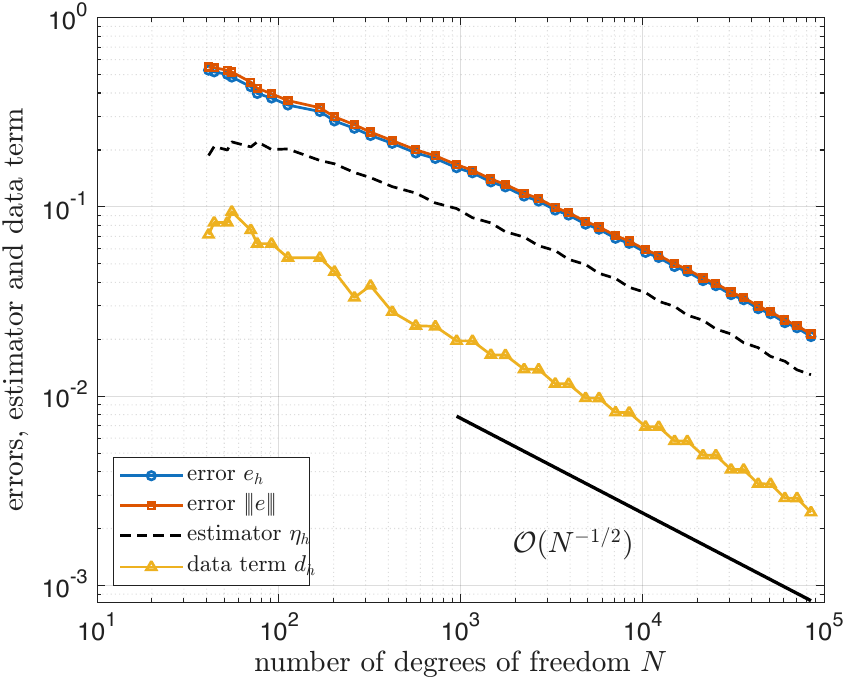}
\caption{Dirichlet data in $H^{1/2+t}(\G_D)$, $t<\tfrac14$: the errors $e_h$
and $\tri e\tri:=\tri(\bsigma-\bsigma_h,u-u_h)\tri$, the estimator $\eta_h$ and the data
term $d_h$ on uniform meshes (left) and on adaptive meshes generated with the
indicators $\eta_K^{2}+d_{h,K}^{2}$ (right).}
\label{fig:roughD}
\end{figure}

On uniform meshes the error, the estimator and $d_h$ decrease like
$h_{\max}^{1/4}$, that is, like $N^{-1/8}$, in agreement with
Theorem~\ref{thm:apriori} and Remark~\ref{rem:oscD} for every $s,t<\nu$, and
$\eta_h/e_h=0.390$ from level $4$ on, see Figure~\ref{fig:roughD}.  On
adaptive meshes generated with the indicators $\eta_K^{2}+d_{h,K}^{2}$, where
$\eta_K^{2}$ collects the terms of $\eta_h^{2}$ on $K$ and $d_{h,K}^{2}$ those
of $d_h^{2}$ on the faces of $K$, the meshes are refined towards $P$ and the
optimal rate is recovered: over the last four of $40$ iterations the error
decreases like $N^{-0.500}$, the estimator like $N^{-0.504}$ and $d_h$ like
$N^{-0.483}$, and from iteration $10$ on $\eta_h/e_h$ lies between $0.57$ and
$0.63$.

\subsection{Approximations that violate the boundary conditions}\label{sec:num_galerkin}
We apply Section~\ref{sec:nonintrusive} to the $P_1$ Galerkin approximation
$u_h^{\rm G}\in V_h$ of \eqref{weak} with the Dirichlet condition imposed by
nodal interpolation of $\phi$, so that $u_h^{\rm G}\neq\phi$ on $\G_D$ in
general; $\bsigma_h^{\rm G}$ is the solution of \eqref{eq:recovery} for
$\hat u_h=u_h^{\rm G}$, and $\eta^{\rm G}:=\hat\eta$.  We report the error
$e_h^{\rm G}$ in the norm of $e_h$, the Dirichlet part $\eta_D^{\rm G}$ of
$\eta^{\rm G}$, and
$\tri e^{\rm G}\tri:=\tri(\bsigma-\bsigma_h^{\rm G},u-u_h^{\rm G})\tri$.  For
the smooth solution of Section~\ref{sec:num_weak},
$\eta^{\rm G}/\tri e^{\rm G}\tri$ lies between $1.037$ and $1.042$ from level
$3$ on, as for the least-squares solution.  For the datum of
Section~\ref{sec:num_rough}, nodal interpolation violates the Dirichlet
condition considerably: on uniform meshes $\eta_D^{\rm G}/e_h^{\rm G}\approx0.45$,
and $\eta^{\rm G}/\tri e^{\rm G}\tri$ lies between $0.856$ and $0.859$ from
level $2$ on.

\begin{figure}[ht]
\centering
\includegraphics[width=0.40\textwidth]{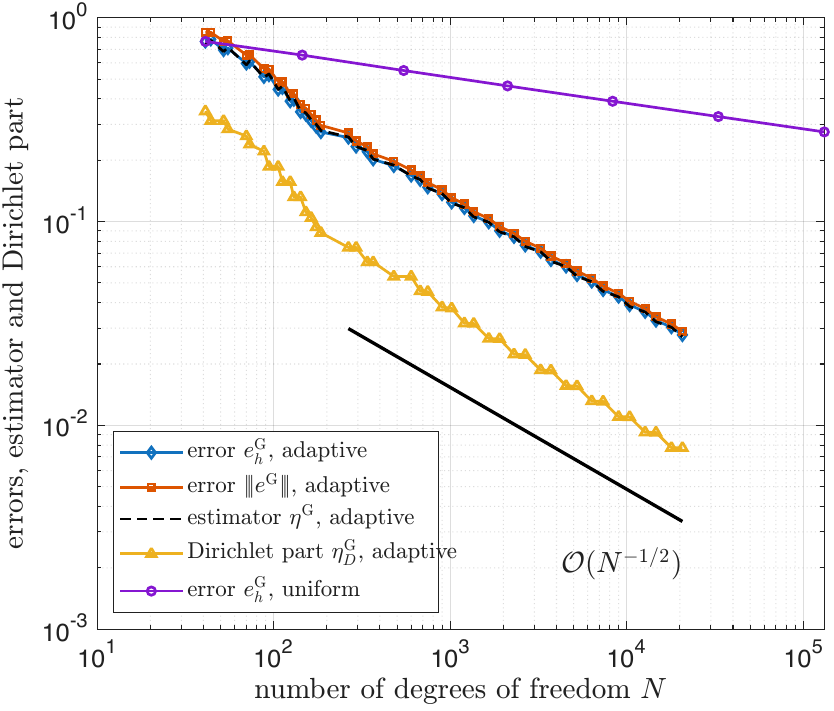}\qquad
\includegraphics[width=0.40\textwidth]{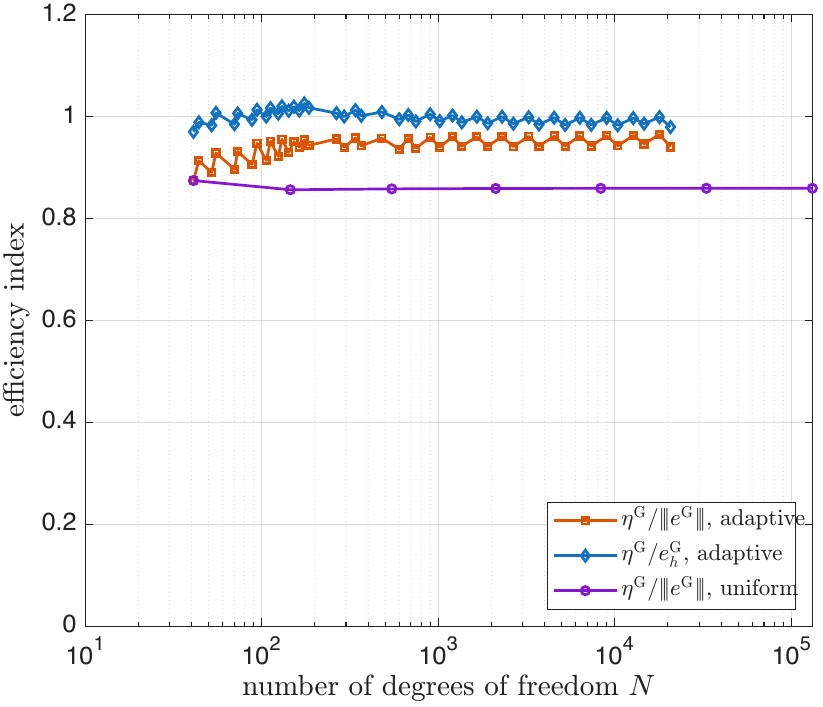}
\caption{Galerkin approximation with nodal Dirichlet data, Dirichlet datum of
Section~\ref{sec:num_rough}: the errors $e_h^{\rm G}$ and $\tri e^{\rm G}\tri$,
the estimator $\eta^{\rm G}$ and its Dirichlet part $\eta_D^{\rm G}$ on
adaptive meshes, and $e_h^{\rm G}$ on uniform meshes (left); efficiency
indices (right).}
\label{fig:galerkin}
\end{figure}

On adaptive meshes generated with the element indicators of $\eta^{\rm G}$,
the error and the estimator decrease like $N^{-0.517}$ and $N^{-0.537}$ over
the last four of $45$ iterations, and from iteration $10$ on
$\eta^{\rm G}/\tri e^{\rm G}\tri$ lies between $0.92$ and $0.97$, see
Figure~\ref{fig:galerkin}.  In agreement with Theorem~\ref{thm:nonintrusive},
the functional is a reliable and efficient error estimator also for
approximations that violate the boundary conditions.

\FloatBarrier
\subsection{The fully discontinuous method}\label{sec:num_dis}
We use the data and meshes of Section~\ref{sec:num_weak}, the broken spaces
$\Sigma_h^{\rm b}\times V_h^{\rm b}$, the boundary weights of
\eqref{def:Jdisc}, and the interior weights $h_F^{-\rho_u}$ and
$h_F^{-\rho_\sigma}$ on the scalar and the flux jumps, so that
$(\rho_u,\rho_\sigma)=(1,1)$ is \eqref{def:Jdisc}.  We report
$e_h^{\rm d}:=\big(\|u-u_h^{\rm d}\|^{2}_{1,h}+\|\bsigma-\bsigma_h^{\rm d}\|^{2}_{H(\divvr,h)}\big)^{1/2}$,
the value $\eta_{\rm d}$ of the functional at the minimizer, and
$\eta_{\rm d}/E_h^{\rm d}$, where $E_h^{\rm d}$ is the middle term of
\eqref{eq:disc_apost} with the same interior weights.

\begin{figure}[ht]
\centering
\includegraphics[width=0.40\textwidth]{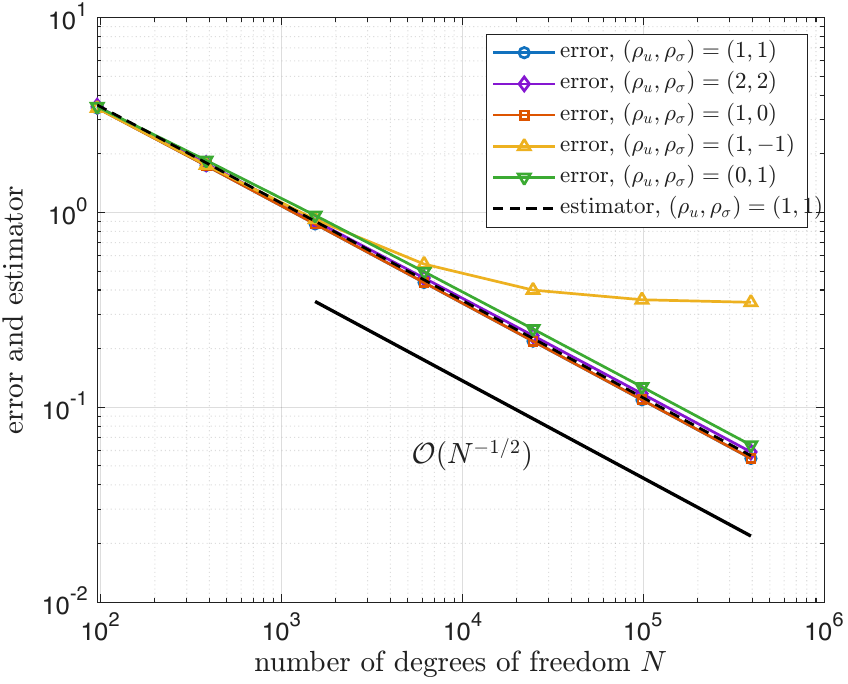}\qquad
\includegraphics[width=0.40\textwidth]{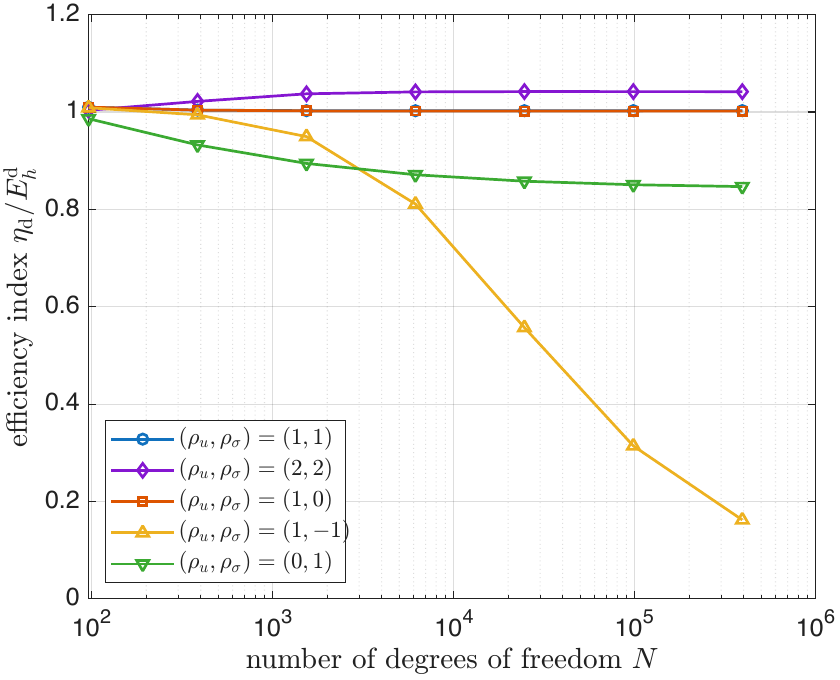}
\caption{The fully discontinuous method on uniform meshes with the interior
weights $h_F^{-\rho_u}$ and $h_F^{-\rho_\sigma}$: the error $e_h^{\rm d}$, and
the estimator $\eta_{\rm d}$ for $(\rho_u,\rho_\sigma)=(1,1)$ (left); the
efficiency index $\eta_{\rm d}/E_h^{\rm d}$ (right).}
\label{fig:dis}
\end{figure}

For $(1,1)$ the error converges with the rate $h_{\max}$ of
Theorem~\ref{thm:disc_apriori}, and $\eta_{\rm d}/E_h^{\rm d}=1.003$ from level
$3$ on, in line with Theorem~\ref{thm:disc_apost}, see Figure~\ref{fig:dis};
the over-penalized pair $(2,2)$ keeps the rate, as stated in
Remark~\ref{rem:disc_overpen}.  With the weight one on either jump, $(1,0)$ or
$(0,1)$, the minimizer still converges with rate one, better than the bound of
order $h_{\max}^{1/2}$ of Remark~\ref{rem:disc_weaker}.  The pair $(1,-1)$
puts the weight $h_F$ of \cite{Bri:26} on the flux jump without the constraint
on its lowest moments; then $e_h^{\rm d}$ stagnates near $0.35$ and
$\eta_{\rm d}/E_h^{\rm d}$ falls to $0.162$ on level $7$, so the optimal rate
and the error equivalence are lost, in line with the bound of order one that
Remark~\ref{rem:disc_weaker} gives for this pair.

The norm equivalence itself fails for the weaker flux weights, as
Proposition~\ref{prop:disc_sharp} states.  For $A=I$, $\bb=\bbeta=\bzero$,
$\mu=0$ and $w_h:=(-\nabla_hI_hg,\,I_hg)$ with $I_h$ the nodal interpolant and
$g(x,y)=\sin(\pi x)$, the volume residuals and the scalar jumps vanish, $w_h$
satisfies the Dirichlet condition, and its flux jumps and its normal trace on
$\G_N$ are of order $h_{\max}$.  The ratio of $J_{\rm dis}(w_h;0)$ to
$\|I_hg\|^{2}_{1,h}+\|\nabla_hI_hg\|^{2}_{H(\divvr,h)}$ is bounded for
$\rho_\sigma=1$ and decreases like $h_{\max}$ and $h_{\max}^{2}$ for
$\rho_\sigma=0$ and $\rho_\sigma=-1$.  That the minimizer with $\rho_\sigma=0$
nevertheless converges is the gap between the norm equivalence and the
convergence of the minimizer described in Remark~\ref{rem:disc_weaker}.  The
spectral condition number of the matrix of the normal equations is smallest for
$\rho_\sigma=1$ among the weights tested.

\FloatBarrier
\subsection{Singular solutions on adaptive meshes}
The adaptive meshes are generated by the fully discontinuous method with
$(\rho_u,\rho_\sigma)=(1,1)$, every interior face term of the indicators being
split equally between the two elements sharing the face; on the same meshes the
problem is also solved with the flux-jump weights $1$ and $h_F$.

\emph{Corner singularity.}  Let
$\O=(-1,1)^{2}\setminus\big([0,1)\times(-1,0]\big)$ be the L-shaped domain,
$A=I$, $\bb=\bbeta=\bzero$, $\mu=0$, $f_1=0$, $\bff_2=\bzero$, and
$u=r^{2/3}\sin(2\theta/3)$ in polar coordinates, with the Dirichlet condition
on the whole boundary; the analysis holds for $\G_N=\emptyset$ with all
Neumann terms omitted.  Over the last four of $26$ iterations the errors for
$\rho_\sigma=1$ and $\rho_\sigma=0$ decrease like $N^{-0.514}$ and
$N^{-0.513}$ and the estimator for $\rho_\sigma=1$ like $N^{-0.503}$, and for
$N\ge10^{4}$ the efficiency indices for both weights lie between $0.54$ and
$0.57$, see Figure~\ref{fig:Lshape}.  For $\rho_\sigma=-1$ the error stagnates
near $0.075$.

\begin{figure}[ht]
\centering
\includegraphics[width=0.40\textwidth]{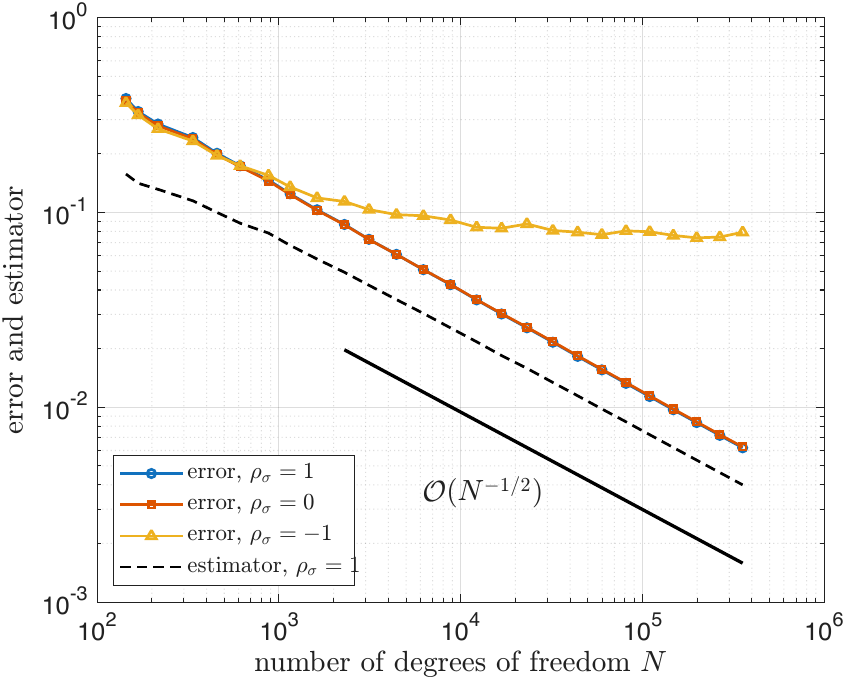}\qquad
\includegraphics[width=0.40\textwidth]{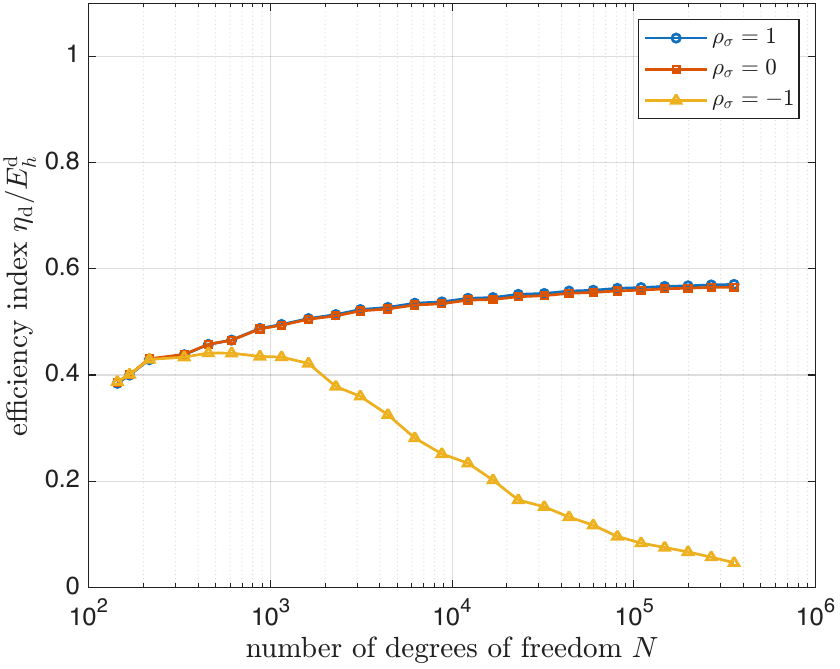}
\caption{L-shaped domain, adaptive meshes: the error $e_h^{\rm d}$ for three
flux-jump weights $h_F^{-\rho_\sigma}$ on the same meshes, and the estimator
$\eta_{\rm d}$ for $\rho_\sigma=1$ (left); the efficiency index
$\eta_{\rm d}/E_h^{\rm d}$ (right).}
\label{fig:Lshape}
\end{figure}

\emph{Discontinuous coefficients.}  Let $\O=(-1,1)^{2}$, $A=aI$ with $a=R$ in
the first and third quadrants and $a=1$ in the second and fourth,
$\bb=\bbeta=\bzero$, $\mu=0$, $f_1=0$ and $\bff_2=\bzero$.  The solution of
Kellogg \cite{Kel:74}, a standard test for robust a posteriori error
estimators on adaptive meshes \cite{CZ:09,CZ:10b}, is $u=r^{\gamma}m(\theta)$
in polar coordinates, with a continuous, piecewise trigonometric function $m$.
We use $\gamma=\tfrac14$ with $R=\cot^{2}(\pi/16)\approx25.27$ and
$\gamma=\tfrac12$ with $R=\cot^{2}(\pi/8)=3+2\sqrt2$, and refer to
\cite{ZhangKel:26} for $m$ and for the relation between $\gamma$ and $R$.  The
Dirichlet condition is imposed on the whole boundary.  Then $u\in H^{1+s}(\O)$
and $\bsigma=-a\nabla u\in H^{s}(\O)^{2}$, with $\gradt\bsigma=0$, exactly for
$s<\gamma\le\tfrac12$.  Since $|\bsigma|$ is of order $r^{\gamma-1}$, the
normal trace of $\bsigma$ on the faces through the origin is in general not
square integrable; the fully discontinuous method and its analysis use only
jumps of discrete functions.  The divergence residual, the Dirichlet term and
the jump terms of \eqref{def:Jdisc} are scaled by the coefficient:
\begin{align*}
J_{\rm w}(\btau_h,v_h;\phi)
:={}&\|a^{-\half}(\btau_h+a\nabla_hv_h)\|_{0,\O}^{2}
+\|a^{-\half}\divvr_h\btau_h\|_{0,\O}^{2}
+\sum_{F\in\cE_D}a_{K_F}h_F^{-1}\|v_h-\phi\|_{0,F}^{2}\\
&+\sum_{F\in\cE_I}\Big(a^{\max}_{F}h_F^{-1}\|[v_h]\|_{0,F}^{2}
+(a^{\min}_{F})^{-1}h_F^{-\rho_\sigma}\|[\btau_h\cdot\bn_F]\|_{0,F}^{2}\Big),
\end{align*}
where $a^{\max}_{F}$ and $a^{\min}_{F}$ are the larger and the smaller value of
$a$ on the two elements sharing $F$.  For $\rho_\sigma=1$ the weights of
$J_{\rm w}$ differ from those of \eqref{def:Jdisc} by factors between $R^{-1}$
and $R$, so the results of Section~\ref{sec:dlsfem} apply to $J_{\rm w}$ with
constants that may depend on $R$; the scaling is meant to reduce this
dependence.  With $(\bsigma_h^{\rm d},u_h^{\rm d})$ now the minimizer of
$J_{\rm w}$, we report
$$
e_h^{\rm w}:=\Big(\|u-u_h^{\rm d}\|^{2}_{0,\O}
+\|a^{\half}\nabla_h(u-u_h^{\rm d})\|^{2}_{0,\O}
+\|a^{-\half}(\bsigma-\bsigma_h^{\rm d})\|^{2}_{0,\O}
+\|a^{-\half}\divvr_h(\bsigma-\bsigma_h^{\rm d})\|^{2}_{0,\O}\Big)^{1/2},
$$
the estimator $\eta_{\rm w}:=J_{\rm w}(\bsigma_h^{\rm d},u_h^{\rm d};\phi)^{1/2}$,
and $\eta_{\rm w}/E_h^{\rm w}$, where $(E_h^{\rm w})^{2}$ is
$(e_h^{\rm w})^{2}$ plus the jump and boundary terms of $J_{\rm w}$ at
$(\bsigma_h^{\rm d},u_h^{\rm d})$.  The meshes are generated by $J_{\rm w}$
with $\rho_\sigma=1$, starting from the mesh of eight triangles of
$(-1,1)^{2}$ whose edges contain the two axes, refined uniformly once.

\begin{figure}[ht]
\centering
\includegraphics[width=0.40\textwidth]{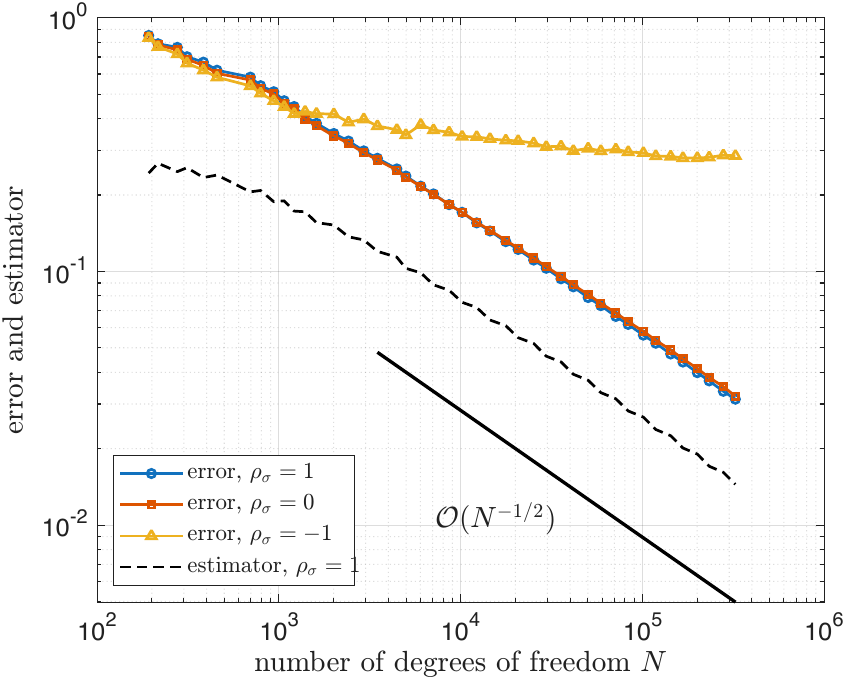}\qquad
\includegraphics[width=0.40\textwidth]{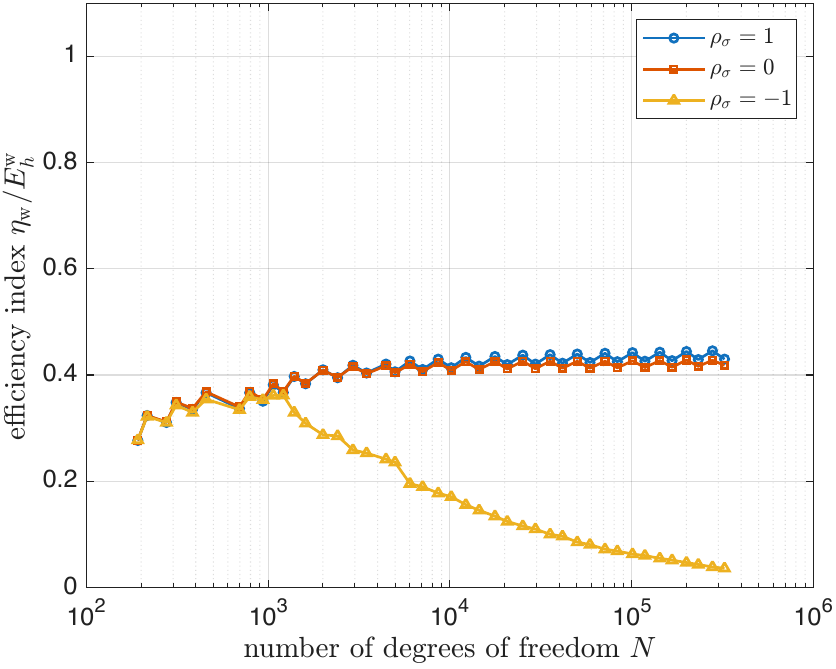}
\caption{Kellogg problem with $\gamma=\tfrac14$, adaptive meshes generated by
$J_{\rm w}$ with $\rho_\sigma=1$: the error $e_h^{\rm w}$ for three flux-jump
weights $h_F^{-\rho_\sigma}$ on the same meshes, and the estimator
$\eta_{\rm w}$ for $\rho_\sigma=1$ (left); the efficiency index
$\eta_{\rm w}/E_h^{\rm w}$ (right).}
\label{fig:kellogg}
\end{figure}

For $\gamma=\tfrac14$ the errors for $\rho_\sigma=1$ and $\rho_\sigma=0$
decrease like $N^{-0.505}$ and $N^{-0.514}$, the estimator for $\rho_\sigma=1$
like $N^{-0.514}$, and for $N\ge10^{4}$ the efficiency indices for both
weights lie between $0.40$ and $0.45$, see Figure~\ref{fig:kellogg}; for
$\gamma=\tfrac12$ the errors decrease like $N^{-0.525}$, the estimator like
$N^{-0.511}$, and the indices lie between $0.52$ and $0.56$.  For
$\rho_\sigma=-1$ the error stagnates in both cases.  The meshes in
Figure~\ref{fig:meshes} are refined radially towards the origin, the structure
that \cite{ZhangKel:26} derives for meshes equidistributing the
coefficient-weighted error.

\begin{figure}[ht]
\centering
\includegraphics[width=0.20\textwidth]{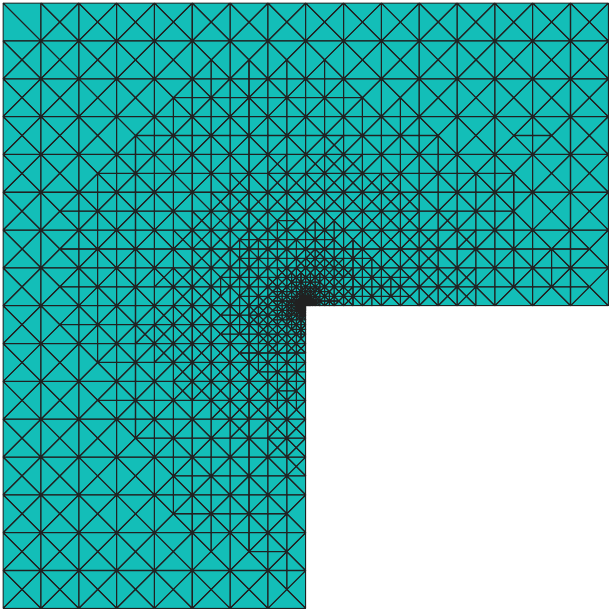}\quad
\includegraphics[width=0.20\textwidth]{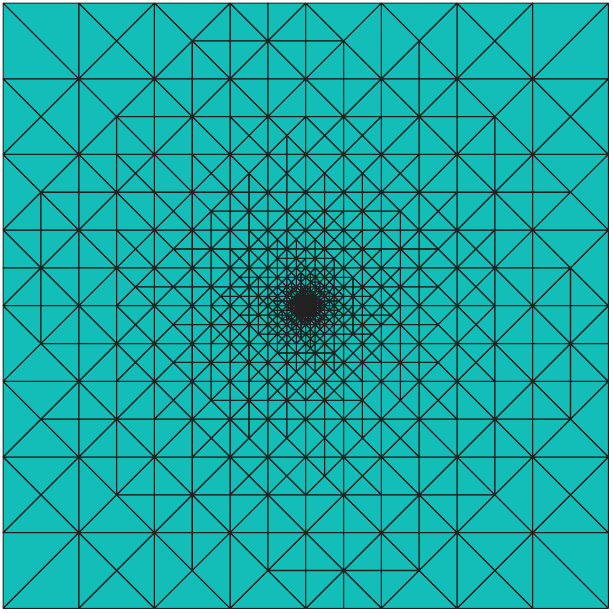}\quad
\includegraphics[width=0.20\textwidth]{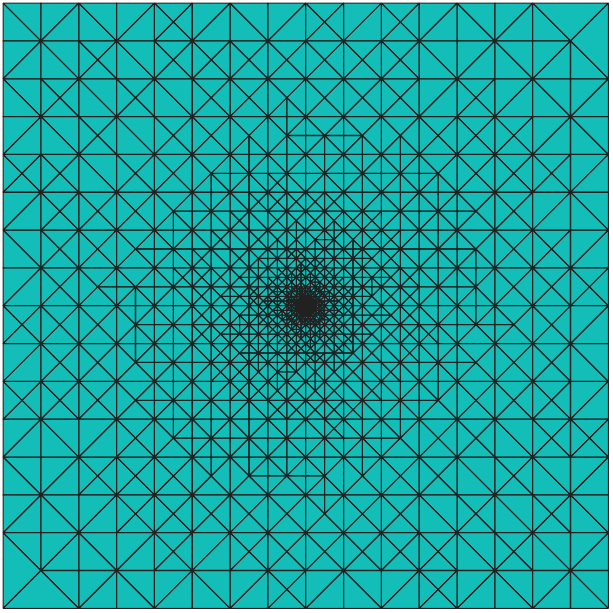}\quad
\includegraphics[width=0.20\textwidth]{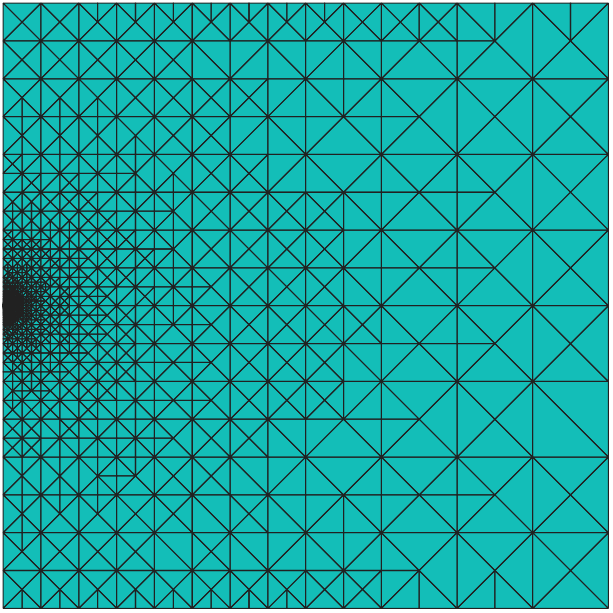}
\caption{Adaptive meshes, from left to right: L-shaped domain ($2038$
elements), Kellogg problem with $\gamma=\tfrac14$ ($2040$ elements) and
$\gamma=\tfrac12$ ($2180$ elements), and the Dirichlet datum of
Section~\ref{sec:num_rough} ($2365$ elements).}
\label{fig:meshes}
\end{figure}

For $\rho_\sigma=1$ the efficiency indices stay bounded in both examples, also
for the Kellogg flux, whose normal trace on the faces through the origin is
not square integrable, consistent with Theorem~\ref{thm:disc_apost}, which uses
jumps of discrete functions only.

\end{document}